\documentclass[12pt,a4paper]{amsart}

\usepackage{pgf,tikz,pgfplots}
\usepackage{xcolor}
\usetikzlibrary{arrows}

\usepackage{todonotes}

\makeatletter
\newcommand{\labeltext}[2]{%
	\@bsphack
	\def\@currentlabel{#1}{\label{#2}}%
	\@esphack
}
\makeatother

\makeatletter
\long\def\unmarkedfootnote#1{{\long\def\@makefntext##1{##1}\footnotetext{#1}}}
\makeatother

\usepackage{amsmath,amssymb}
\usepackage{enumerate,amssymb}
\usepackage{color}
\theoremstyle{plain}
\newtheorem{thm}{Theorem}[section]

\newtheorem{lemma}[thm]{Lemma}

\newtheorem{Problem}[thm]{Problem}

\newtheorem{prop}[thm]{Proposition}
\newtoks\prt

\theoremstyle{definition}

\newtheorem{remark}[thm]{Remark}

\newtheorem{definition}[thm]{Definition}
\newtheorem{example}[thm]{Example}

\def\eqn#1$$#2$${\begin{equation}\label#1#2\end{equation}}

\numberwithin{equation}{section}

\def\D{{\mathcal D}}

\def\diam{\operatorname{diam}}
\def\dist{\operatorname{dist}}

\def\epsilon{\varepsilon}
\def\en{\mathbb N}
\def\er{\mathbb R}

\def\G{\mathcal{G}}

\def\I{\mathcal{F}_p}
\def\Is{\overline{\mathcal{F}_p}^{|\cdot|}}
\def\Iw{\overline{\mathcal{F}_p}^{\rm{w}}}

\def\l{\left}

\def\leqslant{\leq}

\def\mir1{\mathcal L_1}

\def\oint{-\hskip -13pt \int}

\def\phi{\varphi}

\def\sto{\rightrightarrows}

\def\tepsilon{\tilde{\epsilon}}

\newcommand{\RR}{\mathbb{R}}
\newcommand{\N}{\mathbb{N}}

\newtoks\by
\newtoks\paper
\newtoks\book
\newtoks\jour
\newtoks\yr
\newtoks\pages
\newtoks\vol
\newtoks\publ
\def\ota{{\hbox\vol{???}}}
\def\cLear{\by=\ota\paper=\ota\book=\ota\jour=\ota\yr=\ota
\pages=\ota\vol=\ota\publ=\ota}
\def\endpaper{\the\by, {\the\paper},
\textit{\the\jour} \textbf{\the\vol} (\the\yr), \the\pages.\cLear}
\def\endbook{\the\by, \textit{\the\book}, \the\publ.\cLear}
\def\endprep{\the\by, \textit{\the\paper}, \the\jour.\cLear}
\def\endyearprep{\the\by, \textit{\the\paper}, \the\jour, (\the\yr).\cLear}
\def\name#1#2{#2 #1}

\definecolor{dan}{RGB}{120, 0 ,120} 
\newcommand\EEE{\color{black}}
\newcommand{\rone}{\color{black}}

\newcommand{\dc}{\color{dan}}

\usepackage{hyperref}
\def\H{\mathcal{H}}
\def\de0#1{\rule[3pt]{#1}{0.4pt} \hspace{-0.1pt} \rule[3.05pt]{0.05pt}{0.4pt} \hspace{-0.1pt} \rule[3.1pt]{0.05pt}{0.4pt} \hspace{-0.1pt} \rule[3.15pt]{0.05pt}{0.4pt} \hspace{-0.1pt} \rule[3.2pt]{0.05pt}{0.4pt} \hspace{-0.1pt} \rule[3.25pt]{0.05pt}{0.4pt} \hspace{-0.1pt} \rule[3.3pt]{0.05pt}{0.4pt} \hspace{-0.1pt} \rule[3.35pt]{0.05pt}{0.4pt} \hspace{-0.1pt} \rule[3.4pt]{0.05pt}{0.4pt} \hspace{-0.1pt} \rule[3.45pt]{0.05pt}{0.4pt} \hspace{-0.1pt} \rule[3.5pt]{0.05pt}{0.4pt} \hspace{-0.1pt} \rule[3.55pt]{0.05pt}{0.4pt} \hspace{-0.1pt} \rule[3.6pt]{0.05pt}{0.4pt} \hspace{-0.1pt} \rule[3.65pt]{0.05pt}{0.4pt} \hspace{-0.1pt} \rule[3.7pt]{0.05pt}{0.4pt} \hspace{-0.1pt} \rule[3.75pt]{0.05pt}{0.4pt} \hspace{-0.1pt} \rule[3.8pt]{0.05pt}{0.4pt} \hspace{-0.1pt} \rule[3.85pt]{0.05pt}{0.4pt} \hspace{-0.1pt} \rule[3.9pt]{0.05pt}{0.4pt} \hspace{-0.1pt} \rule[3.95pt]{0.05pt}{0.4pt} \hspace{-0.1pt} \rule[4.0pt]{0.05pt}{0.4pt} \hspace{-0.1pt} \rule[4.05pt]{0.05pt}{0.4pt} \hspace{-0.1pt} \rule[4.1pt]{0.05pt}{0.4pt} \hspace{-0.1pt} \rule[4.15pt]{0.05pt}{0.4pt} \hspace{-0.1pt} \rule[4.2pt]{0.05pt}{0.4pt} \hspace{-0.1pt} \rule[4.25pt]{0.05pt}{0.4pt} \hspace{-0.1pt} \rule[4.3pt]{0.05pt}{0.4pt} \hspace{-0.1pt} \rule[4.35pt]{0.05pt}{0.4pt} \hspace{-0.1pt} \rule[4.4pt]{0.05pt}{0.4pt} \hspace{-0.1pt} \rule[4.45pt]{0.05pt}{0.4pt} \hspace{-0.1pt} \rule[4.5pt]{0.05pt}{0.4pt} \hspace{-0.1pt} \rule[4.55pt]{0.05pt}{0.4pt} \hspace{-0.1pt} \rule[4.6pt]{0.05pt}{0.4pt} \hspace{-0.1pt} \rule[4.65pt]{0.05pt}{0.4pt} \hspace{-0.1pt} \rule[4.7pt]{0.05pt}{0.4pt} \hspace{-0.1pt} \rule[4.75pt]{0.05pt}{0.4pt} \hspace{-0.1pt} \rule[4.8pt]{0.05pt}{0.4pt} \hspace{-0.1pt} \rule[4.85pt]{0.05pt}{0.4pt} \hspace{-0.1pt} \rule[4.9pt]{0.05pt}{0.4pt} \hspace{-0.1pt} \rule[4.95pt]{0.05pt}{0.4pt} \hspace{-0.1pt} \rule[5.0pt]{0.05pt}{0.4pt} \hspace{-0.1pt} \rule[5.05pt]{0.05pt}{0.4pt} \hspace{-0.1pt} \rule[5.1pt]{0.05pt}{0.4pt} \hspace{-0.1pt} \rule[5.15pt]{0.05pt}{0.4pt} \hspace{-0.1pt} \rule[5.2pt]{0.05pt}{0.4pt} \hspace{-0.1pt} \rule[5.25pt]{0.05pt}{0.4pt} \hspace{-0.1pt} \rule[5.3pt]{0.05pt}{0.4pt} \hspace{-0.1pt} \rule[5.35pt]{0.05pt}{0.4pt} \hspace{-0.1pt} \rule[5.4pt]{0.05pt}{0.4pt} \hspace{-0.1pt} \rule[5.45pt]{0.05pt}{0.4pt} \hspace{-0.1pt} \rule[5.5pt]{0.05pt}{0.4pt} \hspace{-0.1pt} \rule[5.55pt]{0.05pt}{0.4pt} \hspace{-0.1pt} \rule[5.6pt]{0.05pt}{0.4pt} \hspace{-0.1pt} \rule[5.65pt]{0.05pt}{0.4pt} \hspace{-0.1pt} \rule[5.7pt]{0.05pt}{0.4pt} \hspace{-0.1pt} \rule[5.75pt]{0.05pt}{0.4pt} \hspace{-0.1pt} \rule[5.8pt]{0.05pt}{0.4pt} \hspace{-0.1pt} \rule[5.85pt]{0.05pt}{0.4pt} \hspace{-0.1pt} \rule[5.9pt]{0.05pt}{0.4pt} \hspace{-0.1pt} \rule[5.95pt]{0.05pt}{0.4pt} \hspace{-0.1pt} \rule[6.0pt]{0.05pt}{0.4pt}}	
\def\deb{\mathop{\de0{16pt} ~}\limits}		
\def\debst{\mathop{\hspace{5.2pt} ^* \hspace{-11.2pt} \de0{16pt}~}\limits}	

\title{Non-interpenetration  of  rods derived  by $\Gamma$-limits}

\author{Barbora Bene\v{s}ov\'a, Daniel Campbell, Stanislav Hencl}
\address{Department of Mathematical Analysis, Faculty of Mathematics and Physics,  Charles University,
So\-ko\-lovsk\'a 83, 186~00 Prague 8, Czechia}
\email{\tt \{benesova, campbell, hencl\} @karlin.mff.cuni.cz}

\author{Martin Kru\v{z}\'ik}
\address{ Czech Academy of Sciences, Institute of Information Theory and Automation, Pod vod\'arenskou v\v{e}\v{z}\'i 4, 182~00 Prague 8, Czechia, and  Department of Physics, Faculty of Civil Engineering, Czech Technical University, Th\'{a}kurova 7, 166~29 Prague 6, Czechia   }
\email{\tt kruzik@utia.cas.cz}

\thanks{
This work has been supported by the Czech Science Foundation  through the grants 23-04766S (BB \& MK) and   21-01976S  (DC \& SH) as well as by the Ministry of Education, Youth and Sport of the Czech Republic through the grant LL2105 CONTACT (BB), and 8J22AT017 (MK)}

\date{\today}

\begin{document}

\begin{abstract}
Ensuring non-interpenetration of matter is a fundamental prerequisite when modeling the deformation response of solid materials.  In this contribution, we thoroughly examine how this requirement, equivalent to the injectivity of deformations within  bulk structures, manifests itself in dimensional-reduction problems. Specifically, we focus on the case of rods embedded in a two-dimensional plane. Our results focus on  $\Gamma$-limits of energy functionals that enforce an admissible deformation to be a homeomorphism.  These $\Gamma$-limits are evaluated along a passage from the bulk configuration to the rod arrangement. The proofs rely on the equivalence between the weak and strong closures of the set of homeomorphisms from $\mathbb{R}$ to $\mathbb{R}^2$, a result that is of independent interest and that we establish in this paper, too.

\end{abstract}

\maketitle
\section{Introduction}
Mathematical modeling of (Cauchy) elastic solids aims at predicting the shape of a specimen under the action of applied external forces or boundary conditions.  Hyperelasticity additionally assumes that the first Piola-Kirchhoff stress tensor $T$  has a potential, and so it emphasizes the conservative/nondissipative character of elasticity.  This potential $W$ has a physical meaning of the volume density of energy  stored in the elastic body. In the simplest setting, the stored energy depends only on the gradient of the     \emph{deformation} function $y: \bar \Omega \to \mathbb{R}^n$, where $\Omega\subset\mathbb{R}^n$, {\color{black} with $n=3$ or $n=2$},  is a  reference configuration, i.e., the set spanned by the specimen without applied loads. Typically, it is assumed to be stressless and of zero stored  energy,  i.e., $T=0$ and $W=0$ if $y$ is the identity map.  The formula
\begin{align}\label{PK-ddef}
T(x)=\frac{\partial W(\nabla y(x))}{\partial F}
\end{align}
introduces the relationship  of all mentioned quantities for $x\in\Omega$.  Here $F$ stands for a placeholder  of $\nabla y$. As with any potential, $W$ can only be identified up to an additive constant and it is typically assumed that $W\ge 0$. 

Given the hyperelasticity assumption, the \emph{stable states} of the elastic body are the minimizers of the stored energy functional:
$$
y\mapsto\int_\Omega W(\nabla y) \mathrm{d}x -L(y)
$$
over the set of admissible deformations $\mathcal{A}$. Above, $L$ denotes a functional representing work of external loads.  In many occasions, it is considered linear and continuous in the topology of deformations. 

To ensure that the model is physically sound, several restrictions  must be imposed on the stored energy density $W$ and the set of deformations. Standard properties that can be found, e.g., in the book by Ciarlet \cite{ciarlet}, include that every element in $\mathcal{A}$ is continuous, \emph{injective} and orientation-preserving.  Given that the reference configuration is essentially arbitrary, it also makes sense to require that the inverse of a physical  deformations is also continuous. 
To fix terms,  we define   the set of admissible deformations as 
\begin{align}\label{setA}
\mathcal{A} = \{y: \Omega \to \mathbb{R}^n: y \in W^{1,p}(\Omega; \mathbb{R}^n), y \text{ is a homeomorphism  and  } \mathrm{det}\nabla y > 0 \text{ a.e.~in } \Omega\},\end{align}
 and $p>n$.  Here $W^{1,p}(\Omega;\mathbb{R}^n)$ stands for the usual Sobolev space and the requirement that $\det\nabla y>0$  ensures orientation-preservation.  For the energy density  $W$,  we further impose frame-indifference,  i.e.,  
$$
W(QF) = W(F) \quad \text{for all } Q \in \mathrm{SO}(n) \text{ and } F \in \mathbb{R}^{n\times n}.
$$
It is also natural to impose a blow-up of the energy when  $F$ is approaching infinite extension or compression via
$$
W(F) \to \infty \quad \text{ when } |F| \to \infty \text{ or } \mathrm{det}\,F \to 0.
$$

We wish to emphasize that injectivity is intimately related to non-interpenetration of matter and is thus of the uttermost importance for physical processes. Indeed, an intact material can never undergo a deformation that would map two points from the reference configuration to the same place.

 Although it is important from the modeling point of view, this condition is often disregarded in the mathematical literature. This is because many difficult, and to-date open, questions; see e.g. \cite{ballOpen, sirev, kruzik-roubicek}  are connected with   the constraint on injectivity in variational problems. This is related not only to static problems but also quasistatic and dynamic {\color{black} ones} \cite{kruzik-roubicek} because these are often analyzed by means of time-discrete variational schemes.   We highlight particularly that it is largely unknown whether such functions can be approximated by smooth injective ones or what are traces of maps from  $\mathcal{A}$ on $\partial\Omega$ although partial progress has been achieved in the last years. We refer  e.g. to \cite{CH,IKO, HP}. 

The noninterpretation of the matter seems to be almost completely unexplored in \emph{dimension reduction} which is used to deduce models for a specimen {\color{black} for which} one dimension becomes negligible compared to other ones. This includes  thin films in  the three dimensional space or rods in the plane. We give more background on how a dimension reduction is usually performed in Section \ref{sec:dim-red}, however the essence is to perform an appropriate limit when starting from a bulk model. Up until today, to the best of the authors' knowledge, no such limit has been performed \emph{once starting from a bulk model fully enforcing injectivity} (and thus ensuring non-interpenetration of matter) neither in dimension $n=2$ or $n=3$ unless the limiting model is near to rigid; cf.~\cite{bresciani, BrescianiKruzik}. In some cases, we can recover the limit only in some constrained occasions \cite{DKPS}.

On the top of that, there seems {\it not}  to be  a common agreement how to define  {\it non-interpenetration}  for lower dimensional objects. Clearly,  in lower dimensional structures,  non-interpenetration cannot be set equal to injectivity in this geometry as it is easy to imagine admissible deformations violating injectivity like folding a thin sheet of paper. On the other hand, any crossings in the thin film or rod must correspond to interpenetration. 
In this sense, performing rigorously the above-mentioned limit could shed light on this question, as the class of valid, \emph{admissible and thus non-interpenetrative} deformations should also be a result of the limiting process. 

Within this contribution, we restrict our attention to dimension $n=2$ and perform,  to our best knowledge for the first time, a $\Gamma$-limit of bulk energy functionals that take finite values \emph{only} on $\mathcal{A}$ (and are infinite otherwise) when one dimension of the specimen approaches zero. We perform the limit in the \emph{membrane regime} and thus the most flexible and thus challenging one. In doing so, we arrive at energy functionals that only take finite values on the \emph{closure of homeomorphic functions} in $W^{1,p}((0,1); \mathbb{R}^2)$ so that this set can be identified  as the set of non-interpenetrative deformations of rods. This corresponds to the idea of Pantz \cite{pantz} who proposed to take admissible deformations to be the closure of embeddings isotopic to the reference configuration in the bulk or thin film setting but did not perform a $\Gamma$-Limit. Of course, the topology of the closure is important and a crucial step in our proof is to show that the weak and strong closure of injective functions in $W^{1,p}((0,1); \mathbb{R}^2)$ is the same and equals also the one given by the $C^0$-topology. This result can be shown by an adaptation of the technique of {\color{black} De Philippis and Pratelli} \cite{DPP}.

This paper is built up as follows: In Section \ref{sec:dim-red} we start with providing some background information on dimension reduction, in Section \ref{sec:main} we state then our main results that are then proved in the last two sections of the paper.

\section{Background on dimension reduction}
\label{sec:dim-red}
Many objects in mechanics can be treated in lower-dimensional geometry when one dimension of the object is negligible \rone in comparison  with \EEE the other one or two. For such objects, simplified models can be designed that are often simpler to use in computational modeling \cite{Belik-Brule-Luskin-2001, Belik-Luskin-2004, Belik-Luskin-2006, Friedrich-Kruzik-Valdman}.  Analytically, on the static level, this is built up on the fundamental results of Friesecke, James,  and M\"uller  \cite{friesecke.james.mueller1,friesecke.james.mueller2}  that showed that it is possible to obtain such models as $\Gamma$-limits \cite{braides} of bulk models. 

Let us, at this point, give a short overview of the procedure and the results obtained to date {\color{black} by restricting to the case of planar rods}. As we consider a static situation here (and this is also the setting of \cite{friesecke.james.mueller1,friesecke.james.mueller2}), the analysis starts from a bulk energy of the form
\begin{equation}
 \label{eq-dimRed1}   
 {\color{black} \tilde{y_h} \mapsto \int_{\Omega_h} W(\nabla\tilde{y_h}(\tilde{x})) \mathrm{d}\tilde{x} \qquad \text{ with } \Omega_h:=(0,1)\times(-h/2,h/2).}
\end{equation}

\rone 
 It is more convenient to work on a fixed domain {\color{black} $\Omega:=(0,1)\times(-1/2,1/2)$}. We perform the change of variables
\begin{align}\label{phi}\phi_h:\Omega\to\Omega_h,\quad\phi_h(x):=(x_1,hx_2)\quad {\color{black} \text{ for }(x_1,x_2) \in (0,1) \times (-1/2,1/2).}\end{align}
Setting $y_h:=\tilde y_h\circ \phi_h$, and 
\begin{align}\label{nablah} \nabla_h v :=\Big(\partial_1v  \big|\frac{\partial_2 v }{h}\Big)\end{align} for every  $v\in W^{1,p}(\Omega;\mathbb{R}^2)$  
  we see that \eqref{eq-dimRed1} takes the form \
\begin{equation}
     \label{eq-dimRed2} 
     \mathcal{E}_h(y) = h \int_{\Omega} W(\nabla_h y_h( x)) \mathrm{d}{x}.\EEE 
\end{equation}
\EEE

As the volume of the bulk object vanishes  with the shrinking thickness, it is natural that the same happens for its energy \eqref{eq-dimRed1}, which is even more apparent in \eqref{eq-dimRed2}.  Hence, we rescale   the energy $\mathcal{E}_h$ just to compensate for it. This scaling  is the so-called   \emph{membrane} regime in which one considers the limit $\mathcal{E}_h/h$ and thus only balances  the vanishing volume.  This allows the incorporation of stretching or rotations of the midplane. \rone The rigorous derivation of the membrane regime for three dimensional energies with polynomial growth is due to LeDret and Raoult in \cite{LeDret-Raoult}. \EEE When taking $\alpha > 1$ and considering $\mathcal{E}_h/h^\alpha$ one obtains a whole hierarchy of more restrictive models corresponding to non-linear bending {\color{black} von K\'arm\'an type theories }(see \cite{friesecke.james.mueller2}).   \rone Hornung \cite{Hornung} studied properties of almost minimizers of $\mathcal{E}_h/h^\alpha$  for $\alpha\ge 3$ and showed that they are invertible almost everywhere possibly apart of a thin boundary layer whose thickness is estimated. \EEE   In this work, however, we stick to the case $\alpha =1$. 

To summarize, the thin-film model shall be obtained as an appropriate limit of $\mathcal{E}_h/h$ as $h \to 0$.  As  we work in a variational setting, the limiting process should also transfer information on minimizers and minimizing sequences, which can be ensured by considering the $\Gamma$-convergence (see, e.g., \cite{braides, dalmaso}). 
 Before we state the definition of the $\Gamma$-convergence tailored to our setting we define for $p\ge 1$ an averaging operator $\pi: L^p(\Omega;\mathbb{R}^2)\to L^p((0,1);\mathbb{R}^2)$ such that for a.e.~$0<t<1$
$$
\pi(u)(t) := \int_{-1/2}^{1/2} u(t,x) \mathrm{d} x.
$$

\begin{definition}
\label{GammaConvDef}
Take  $p \in (1,\infty)$.   We say that a sequence of functionals  $\{J_h\}_h$, where $J_h: W^{1,p}(\Omega;\mathbb{R}^2) \to \mathbb{R}$ $\pi$-$\Gamma$-converges to $J: W^{1,p}((0,1);\mathbb{R}^2) \to \mathbb{R}$ if
\begin{itemize}
    \item for every sequence $\{y_h\}_{h>0} \subset W^{1,p}(\Omega; \mathbb{R}^2)$ such that $\pi(y_h) \to y$ in $L^p((0,1); \mathbb{R}^2)$
    $$
    J(y) \leq \liminf_{h\to 0} J_h(y_h);
    $$
    \item for every $u \in W^{1,p}((0,1); \mathbb{R}^2)$ there exists a sequence of functions $\{u_h^\mathrm{rc}\}_{h>0} \subset  W^{1,p}(\Omega, \mathbb{R}^2)$ such that  $\pi(u_h^\mathrm{rc}) \to u$ in $L^p((0,1); \mathbb{R}^2)$ and 
    $$
 J(u) \geq \limsup_{h\to 0} J_h(u_h^\mathrm{rc}).
$$
\end{itemize}
\end{definition}

Hence,  the definition of $\Gamma$-convergence includes two parts: {\color{black} a liminf-inequality} and the construction of a so-called recovery sequence $u_h^\mathrm{rc}$.   

Under the crucial assumption of $p$-growth (for $p \in (1,\infty)$), i.e., 
$$
C_1(1+|F|^p) \leq W(F) \leq C_2 (1+|F|^p) \qquad \text{ for some } C_1,C_1 > 0,
$$
it is proved in \cite{friesecke.james.mueller1,friesecke.james.mueller2} that $\frac{1}{h}\mathcal{E}_h$ $\pi$-$\Gamma$-converges to $\mathcal{E}$ with
$$
\mathcal{E}(y)  = \int_0^1\mathcal{C}  \min_{\xi \in \mathbb{R}^2} W(y'(x)| \xi) \mathrm{d} x,
$$
\rone 
where $\mathcal{C}\min_{\xi \in \mathbb{R}^2} W(\cdot| \xi) $ denotes the convex envelope  of $\min_{\xi \in \mathbb{R}^2} W(\cdot | \xi) $  which appears in the limit energy. \EEE  Namely,  $\Gamma$-limits are always weakly lower semicontinuous, or, in other words, the $\Gamma$ limiting process always involves a relaxation process, too.  

\begin{remark}[Higher dimensions]
Let us note that the original results in  \cite{friesecke.james.mueller1,friesecke.james.mueller2} hold in a more general setting than presented here. Indeed, besides investigating a larger class of convergence regimes, there the bulk energy was defined for functions in $n=3$; i.e. $\mathcal{E}_h:W^{1,p}(\omega \times (-1/2)\times(1/2); \mathbb{R}^3)$ with $\omega \subset \mathbb{R}^2$. Then the limiting energy differs in the relaxation means; i.e.
$$
\mathcal{E}(y) = \int_\omega Q \min_{\xi \in \mathbb{R}^3} W(\nabla_p y(x) | \xi) \mathrm{d} x,
$$
$Q$ is the quasiconvex hull, that is the largest quasiconvex function (in the sense of Morrey \cite{morrey-orig, morrey}, see also \cite{dacorogna}) below $A \mapsto \min_{\xi \in \mathbb{R}^3} W(A| \xi)$.
\end{remark}

The result has been further generalized by Anza-Hafsa and {\color{black} Mandallena} \cite{AnzaHafsa1} (see also \cite{AnzaHafsa2}), building on the results of Ben Belgacem \cite{belgacem} and most importantly the approximation results of Gromov and Eliashberg \cite{GromovEliasberg} in order to incorporate local injectivity so that they could generalize the growth of $W$ to 
\begin{align*}
    &W(F)=+\infty \text { if and only if } \operatorname{det} F \leqslant 0;\\
    &\text{for every $\delta>0$, there exists $c_\delta>0$ such that for all $F \in \mathbb{R}^{2 \times 2}$, if $\operatorname{det} F \geqslant \delta$ then} \\
    &W(F) \leqslant c_\delta\left(1+|F|^p\right),
\end{align*}
so that the local injectivity and the growth of the energy for infinite compression could be included. Here, again the $\Gamma$-limit is given by $\mathcal{E}$ as above, so remains unchanged. This may seem surprising at first but the lower dimensionality of the object allows to approximate (strongly) any function in $W^{1,p}((0,1);\mathbb{R}^2)$ by locally injective functions (see \cite{GromovEliasberg}) so that, in other words, local injectivity is not a strong restriction in this context.

Imposing global injectivity in dimension reduction, however, stays largely open to date. This issue has been addressed only in a couple of works that obtained partial results. Olbermann and Runa \cite{OlbermannRuna} propose a definition of interpenetration based on the Brouwer degree in the thin film setting. They prove that in several scalings (but not the membrane scaling considered here), interpenetrative deformations in the thin film can only be limits of interpenetrative bulk deformations but do not give a full characterization of the $\Gamma$-Limit. This full $\Gamma$-Limit has been found by Bresciani \cite{bresciani} (see also \cite{BrescianiKruzik}), however only in a very rigid scaling corresponding to the linear von K\'{a}rm\'an plate theory. 

In some other works \cite{MarianoMucci}, non-interpenetration of the thin film is characterized by introducing an ``artificial thickness''  so the bulk theory on non-interpenetration is applicable. In the language of $\Gamma$-convergence this means that a construction of a recovery sequence is prescribed and the deformation of the thin film is said to be non-interpenetrative if this prescription yields an injective bulk deformation.

\section{Main Results and Discussion}
\label{sec:main}
As in Section \ref{sec:dim-red}, we take the reference configuration to be $\Omega_h:=(0,1)\times(-h/2,h/2)$, where $h>0$ denotes the thickness of the specimen.  Moreover,  $\Omega:=\Omega_1$.   We then introduce a family of functionals
$$
I_h(\tilde{y}_h): = 
\begin{cases}
\int_{\Omega_h }W (\nabla \tilde{y}(\tilde{x}))\, \mathrm{d} \tilde{x} &\quad \text{ if } \tilde{y}_h \in  \mathcal{A}_h, \\
+\infty &\quad \text{ else, }
\end{cases}
$$
that enforces injectivity, i.e. the functional can only be finite on injective deformations, and our aim is to pass to the limit $h\to 0^+$.
 Here $\mathcal{A}_h$ is defined as in \eqref{setA} with $\Omega_h$ instead of $\Omega$. We perform the change of variables as in \eqref{phi} and \eqref{nablah}
which yields 
    \begin{align}
    \label{bulkFunct}
     J_h(y_h) :=\frac{1}{h}I_h(\tilde y_h):= \begin{cases}
\int_\Omega W (\nabla_h y_h(x) )\, \mathrm{d} {x} &\quad \text{ if } y_h \in \mathcal{A}, \\
+\infty &\quad \text{ else, }
\end{cases}
 \end{align}

 where we already introduced {\color{black} the class $\mathcal{A}$} in \eqref{setA}.

 \begin{remark}[Existence of minimizers in the bulk]
 Let us stress that we do not investigate within this paper the existence of minimizers if $J_h$ in \eqref{bulkFunct}, but just concentrate on the $\Gamma$-limit. Hence, we do not put any convexity (or generalized convexity) assumptions on $W$ and thus completely avoid the hard questions of the existence of minimizers on injective function (see, e.g., \cite{ballOpen}) outside the class of polyconvex functionals. As the $\Gamma$-limit procedure, however, always also contains relaxation \cite{braides}, we know that the obtained lower dimensional object will possess minimizers.
 \end{remark}

 We will now investigate the $\Gamma$-limit as $h \to 0$ of the sequence of functionals in \eqref{bulkFunct} as introduced in Section \ref{sec:dim-red}. What we expect is that it is particularly the \emph{global non-interpenetration} (or in other words global injectivity) that needs to translate to the lower-dimensional object. The intuition is that a thin bulk object can touch itself on the boundary upon deformation but it cannot intersect. Once the thinning process continues ad infinitum, the object consists only of a boundary that can touch. Thus, injectivity as a characterizer of non-interpenetration has to be given up. On the other hand, the property that intersections are not possible should withstand the limiting procedure. As we shall see later, in Subsection  \ref{sec:non-interpen} an explicit characterization of this ``non-crossing'' property is not easy, but the following definition seems straightforward: 
 
 \begin{definition}
 \label{def-interpen-rod}
  Let $y \in W^{1,p}((0,1); \mathbb{R}^2)$ with $p > 1$. Then, we say that $y$ is a \emph{interpenetrative} if for every  bounded sequence $\{ y_k\} \subset W^{1,p}((0,1); \mathbb{R}^2)$ such that $\lim_{k\to\infty}\|y_k-y\|_{\mathcal{C}^0((0,1);\mathbb{R}^2)}=0$  there is  $K \in \mathbb{N}$ such that $\forall k \geq K$ $y_k$ is not a homeomorphisms. 

  {\color{black}Conversely, we say that $y \in W^{1,p}((0,1); \mathbb{R}^2)$ with $p > 1$ is \emph{non-interpenetrative} if there exists a bounded sequence $\{\tilde y_k\} \subset W^{1,p}((0,1); \mathbb{R}^2)$ of homeomorphisms such that $\lim_{k\to\infty}\|\tilde y_k- y\|_{\mathcal{C}^0((0,1); \mathbb{R}^2)}=0$.}

\end{definition}

\rone Here and in the following, we use the shorthand "homeomorphism" for a homeomorphism onto its image. \EEE

Indeed, the definition formalizes the intuitive idea that if the rod intersected itself - as opposed to mere touching - then this needs to be the case also for any other function in $W^{1,p}((0,1); \mathbb{R}^2)$ that is close enough to the original intersecting map. To prove, however, that this definition of non-interpenetration is physically correct, we need to show that $\Gamma$-limits of the functional $J_h$ can only be finite on the class non-interpenetrative functions. 

We will show that this is indeed the fact and claim that the $\Gamma$-limit of \eqref{bulkFunct} is the following functional
\begin{equation}
    \label{rodFunct}
    J(y(x)):= \begin{cases}
\int_0^1 \mathcal{C}\min_{\xi \in \mathbb{R}^2} W(y'(x)| \xi) \mathrm{d} {x} &\quad \text{ if } {y} \text{ is non-interpenetrative}, \\
+\infty &\quad \text{ else, }
\end{cases}
\end{equation}
where $\mathcal{C}\min_{\xi \in \mathbb{R}^2} W(\cdot| \xi)$ is the convex envelope of $\min_{\xi \in \mathbb{R}^2} W(\cdot| \xi)$.

This is formalized in the following theorem: 

\begin{thm}
\label{thm-dimRed}
Let $p \in \dc [\EEE 2, \infty)$ and  suppose that $W$ in \eqref{bulkFunct} is a continuous function on $\mathbb{R}^{2 \times 2}$ with $\det (\cdot) > 0$ while
$$
W(F)=+\infty \text { if and only if } \operatorname{det} F \leqslant 0
$$
and satisfies the following growth condition: there exits a $c > 0$ and for every $\delta>0$, there exists $c_\delta>0$ such that for all $F \in \mathbb{R}^{2 \times 2}$  with $\operatorname{det} F \geqslant \delta$ it holds
\begin{align*}
    c|F|^p \leqslant W(F) \leqslant c_\delta\left(1+|F|^p\right).
\end{align*}

Then, the functionals $J_h$ from \eqref{bulkFunct} $\pi{-}\Gamma${\color{black}-}converge (in the sense of Definition \ref{GammaConvDef}) to $J$ given in \eqref{rodFunct}.
\end{thm}

\rone 
\begin{remark}
Let us remark that in the proof of Theorem \ref{thm-dimRed} ahead, we will also show the following equi-coercivity property. If  $\{y_h\}\subset W^{1,p}(\Omega; \mathbb{R}^2)$ is such that $J_h(y_h) \leq C$ for some constant $C$, then (at least for a non-relabelled subsequence)
\begin{align}\label{min}
y_h \rightharpoonup y \qquad \text{ for some } y \in W^{1,p}(\Omega; \mathbb{R}^2),
\end{align}
such that $y$ is constant in the second variable. 

Further,   the Fundamental Theorem of  Gamma-convergence \cite[Thm.~7.8]{dalmaso} states that  if $\lim_{h\to 0}(J_h(y_h)-\inf_{\mathcal{A}}J_h)=0$ for a sequence $\{y_h\}_{h>0}$ satisfying \eqref{min} then $y$ minimizes $J$.
\end{remark}
\EEE

The proof of Theorem \ref{thm-dimRed} relies crucially on the fact that the set non-interpenetrative deformations can {\color{black} be} characterized as the closure of homeomorphisms in $W^{1,p}((0,1); \mathbb{R}^2)$ with respect to the weak and strong topologies and that these are the same. Indeed, we shall use the availability of the first closure while proving the "liminf"-condition and the latter to provide a recovery sequence. We summarize this statement in the following theorem:

\begin{thm}\label{NotParticullarlyDeep}
	Let $p\in(1,\infty)$ and $y \in W^{1,p}((0,1), \er^2)$. The following conditions are equivalent
	\begin{enumerate}
            \item The map $y \in W^{1,p}((0,1), \er^2)$ is non-interpenetrative.
            \item There exists a bounded sequence of homeomorphisms  $\{y_k\} \in W^{1,p}((0,1); \mathbb{R}^2)$ such that $y_k \rightharpoonup y \in  W^{1,p}((0,1); \mathbb{R}^2)$.
		\item There exists a bounded sequence of homeomorphisms  $\{\tilde y_k\} \in W^{1,p}((0,1); \mathbb{R}^2)$ such that $\tilde y_k \to y \in  W^{1,p}((0,1); \mathbb{R}^2)$. Moreover, the approximating sequence can be chosen in such a way that all  $\tilde y_k$   are piecewise affine.
	\end{enumerate}   
\end{thm}

\rone
Let us remark that Theorem \ref{NotParticullarlyDeep} is strictly restricted to the planar case. Analogous characterizations have been obtained also for planar bulk deformations (\cite{iwan,DPP}), but a generalization to deformations $y: \mathbb{R}^2 \to \mathbb{R}^3$ or even $y: \mathbb{R}^3 \to \mathbb{R}^3$ this seems to be widely open at this moment.

\begin{Problem}
\label{prob-char0}
Can a characterization analogous to Theorem \ref{NotParticullarlyDeep} be also provided for deformations describing thin films, i.e. $y: \mathbb{R}^2 \to \mathbb{R}^3$?
\end{Problem}

\EEE

\label{sec:non-interpen}
The definition of non-interpenetration in Definition \ref{def-interpen-rod} is, albeit natural, quite implicit. Indeed, it relies on constructing an approximating sequence, which can be a challenging task in general situations. So, it is a natural question if a more explicit and an easy-to-check condition can be given as a characterization of non-interpenetration.  Hence,  we formulate the following open problem:

\begin{Problem}
\label{prob-char}
Is there a characterization of non-interpenetration (as defined in Definition \ref{def-interpen-rod}) that can be easily checked be found? 
\end{Problem}

Within this paper, we do not solve Problem \ref{prob-char}, but we shall at least explore some ideas that might look tempting but turn out to be insufficient. The discussion thus will indicate that still non-interpenetration on lower dimensional objects will need more exploration.  We do, however, refer the interested reader to the paper \cite{DC} for a discussion on related topics in the $\er^2 \to \er^2$ case.

One obvious direction that one could try is to look for inspiration in the bulk setting. There, the so-called Ciarlet-Ne\v{c}as condition (see \cite{ciarlet-necas}) that reads as follows
\begin{align}\label{eq:ciarlet-necas}
\int_{\Omega_h} {\rm det }\,\nabla \tilde y_h(x)\,\mathrm{d} x\leq
\mathcal{L}^n(y_h(\Omega_h))
\end{align}
can be imposed on $\tilde y_h\in\mathcal{A}_h$ in order to ensure injectivity of deformations almost everywhere in the deformed configuration $\tilde y_h(\Omega_h)$.
This condition, however, provides only scant information if we pass to the lower-dimensional objects.  Indeed,  recall that  $\Omega=(0,1) \times (-1/2;1/2)$. As outlined in Section above, we perform the change of variables {\color{black}$\phi_h:\Omega\to\Omega_h$} and set $y:=y_h\circ \phi_h$ as well as $$\nabla_h v :=\Big(\partial_1v  \big|\frac{\partial_2 v }{h}\Big)$$ for every for every $v\in W^{1,p}(\Omega;\mathbb{R}^2)$. The limit passage for $h\to 0$ in \eqref{eq:ciarlet-necas} leads to  
\begin{align*}
\int_0^1 (\partial_1 y_h)^\perp\cdot b\, {\rm d} S\le
\lim_{h\to 0} \frac{\mathcal{L}(y_h(\Omega_{h})}{h} 
\end{align*}
where $b$ is  a {\it Cosserat} vector obtained as  $b=\lim_{h\to 0} h^{-1} \partial_2 y_h$ in $ W^{1,p}(\Omega;\er^2)$ and for any $a \in \mathbb{R}^2$, we have that $(a_1,a_2)^\perp {\color{black} := } (-a_2, a_1)$.
In particular, if $b= (\partial_1 y)^\perp/|(\partial_1 y)^\perp|$, i.e., it is the unit normal vector to the film in the deformed configuration,
and if $\lim_{h\to 0} \mathcal{L}^2(y_h(\Omega_{h}))/h=\mathcal{H}^1(y((0,1))$ we get
\begin{align}\label{cn2D-spec}
\int_\omega |(\partial_1 y)^\perp|\,{\rm d}S\le \mathcal{H}^1(y((0,1))).
\end{align}

The left-hand side of \eqref{cn2D-spec} is the length of the film calculated by the change-of-variables formula while the right-hand side is 
the measured length. Hence, \eqref{cn2D-spec} is violated by a {\it
  folding} deformation, which  should be  admissible  among the family of realistic thin-film deformations,  while  \eqref{cn2D-spec} 
is satisfied if the film crosses itself,  which   violates
non-self-interpenetration of matter and is hence not admissible. \rone Notice that this observation was already made in \cite{DKPS}.\EEE

Another tempting path may be to characterize non-interpenetration by using degree theory at least in the case when the limit deformation $y$ satisfies that $y(0)=y(1)$ so we can assume that $y$ is defined on the unit circle in $\mathbb{R}^2$ instead of $(0,1)$. Now, every injective continuous embedding $y$ of the circle into the plane, decomposes the plane into two domains: a bounded domain where the degree with respect to \rone $y$ \EEE  is one (or minus one) and the unbounded domain where the degree is zero. Since the degree is stable under uniform convergence, one could ask whether the set of non-interpenetrative functions is (at least in this special setting)  identical to
	$$
	\D:=\bigl\{y \in W^{1,p}((0,1),\er^2):\ y(0)=y(1)\text{ and }\deg(x,y)\in\{0,1\}\text{ for a.e. }x\in\er^2 \bigr\}. 
	$$

But this is not the case as we show in Example~\ref{horny}: 

\begin{example}\label{horny} (Horned devil) 
\rone
The function $y$ as depicted on the right-hand side of Figure \ref{devil} cannot be approximated by injective curves. Here, it is important to notice that it is not only the graph of the image that we need to approximate, but indeed the function itself which includes the way how the graph is ``run through". Therefore, we also marked the image of the points $A, \ldots, F$ in the image.    

Keeping the above in mind, we provide on the left one member of a possible sequence of approximating  functions for $y$, which is non-injective since one of the ``horns" intersects the vertical connection. It would be possible to fix this non-injectivity produced by the right ``horn" but only at the expense that the left one would now intersect the vertical connection. Thus, we conclude that $y$ as given cannot be approximated by injective functions.

It is not difficult to check though that $\deg(\cdot,y)=1$ both in the rectangle in the target and in both ``horns''. This is because, roughly speaking, the degree ``sees'' only the graph of the function and not the ``run through". \EEE
	
	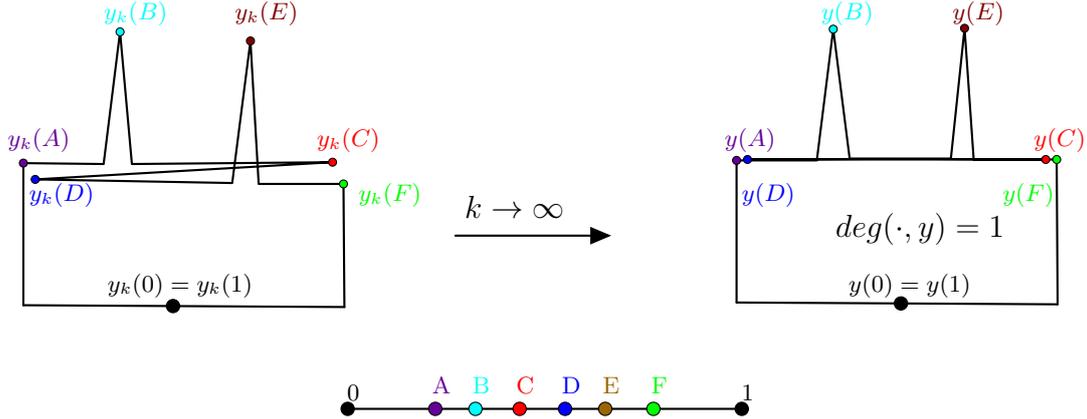
\begin{figure} [h t p]
		\definecolor{zzwwqq}{rgb}{0.6,0.4,0.}
\definecolor{qqffqq}{rgb}{0.,1.,0.}
\definecolor{yqqqqq}{rgb}{0.5019607843137255,0.,0.}
\definecolor{qqqqff}{rgb}{0.,0.,1.}
\definecolor{ffqqqq}{rgb}{1.,0.,0.}
\definecolor{qqffff}{rgb}{0.,1.,1.}
\definecolor{wwqqzz}{rgb}{0.4,0.,0.6}
\begin{tikzpicture}[line cap=round,line join=round,>=triangle 45,x=0.6cm,y=0.6cm]
\clip(-0.1,-4.7) rectangle (24.1,5.2);
\draw [line width=0.8pt] (0.56,-1.74)-- (0.56,1.4)-- (2.32,1.38)-- (2.68,4.3)-- (2.94,1.38)-- (7.34,1.42)-- (0.82,1.04)-- (5.14,0.96)-- (5.54,4.1)-- (5.72,0.94)-- (7.58,0.94)-- (7.6,-1.78);
\draw [line width=0.8pt] (0.56,-1.74)-- (7.6,-1.78);
\draw [line width=0.8pt] (16.2,-1.68)-- (16.2,1.46)-- (17.96,1.46)-- (18.32,4.36)-- (18.68,1.5)-- (22.98,1.48)-- (16.44,1.48)-- (20.9,1.5)-- (21.2,4.38)-- (21.4,1.48)-- (23.22,1.48)-- (23.24,-1.72);
\draw [line width=0.8pt] (16.2,-1.68)-- (23.24,-1.72);
\draw [->,line width=0.8pt] (10.03335334903318,-0.20269468864645065) -- (13.45335334903318,-0.20269468864645065);
\draw (10.0,0.9206869488032271) node[anchor=north west] {$k\to\infty$};
\draw (18.13235299036017,0.541374484500099) node[anchor=north west] {$deg(\cdot,y)=1$};
\draw [line width=0.8pt] (7.671940157880253,-4.030313081547865)-- (16.313761685506993,-4.030313081547865);
\begin{scriptsize}
\draw [fill=wwqqzz] (0.56,1.4) circle (1.5pt);
\draw[color=wwqqzz] (0.9020842993905577,1.9) node {$y_k(A)$};
\draw [fill=qqffff] (2.68,4.3) circle (1.5pt);
\draw[color=qqffff] (3.026234099488078,4.7) node {$y_k(B)$};
\draw [fill=ffqqqq] (7.34,1.42) circle (1.5pt);
\draw[color=ffqqqq] (7.6917774104165595,1.9) node {$y_k(C)$};
\draw [fill=qqqqff] (0.82,1.04) circle (1.5pt);
\draw[color=qqqqff] (1.4,0.7) node {$y_k(D)$};
\draw [fill=yqqqqq] (5.54,4.1) circle (1.5pt);
\draw[color=yqqqqq] (5.890043204976699,4.7) node {$y_k(E)$};
\draw [fill=qqffqq] (7.58,0.94) circle (1.5pt);
\draw[color=qqffqq] (8.6,0.7) node {$y_k(F)$};
\draw [fill=wwqqzz] (16.2,1.46) circle (1.5pt);
\draw[color=wwqqzz] (16.501309393856722,1.9) node {$y(A)$};
\draw [fill=qqffff] (18.32,4.36) circle (1.5pt);
\draw[color=qqffff] (18.625459193954242,4.7) node {$y(B)$};
\draw [fill=ffqqqq] (22.98,1.48) circle (1.5pt);
\draw[color=ffqqqq] (23.291002504882723,1.9) node {$y(C)$};
\draw [fill=qqqqff] (16.44,1.48) circle (1.5pt);
\draw[color=qqqqff] (16.9,0.7) node {$y(D)$};
\draw [fill=yqqqqq] (21.2,4.38) circle (1.5pt);
\draw[color=yqqqqq] (21.508233922658018,4.7) node {$y(E)$};
\draw [fill=qqffqq] (23.22,1.48) circle (1.5pt);
\draw[color=qqffqq] (22.6,0.7) node {$y(F)$};
\draw [fill=black] (3.842715600443314,-1.758651793184337) circle (2.5pt);
\draw[color=black] (4.0,-1.3) node {$y_k(0) = y_k(1)$};
\draw [fill=black] (19.804484803224867,-1.7004800272910503) circle (2.5pt);
\draw[color=black] (20.0,-1.3) node {$y(0) = y(1)$};
\draw [fill=black] (7.671940157880253,-4.030313081547865) circle (2.5pt);
\draw[color=black] (7.79608833809992,-3.678476680872202) node {0};
\draw [fill=black] (16.313761685506993,-4.030313081547865) circle (2.5pt);
\draw[color=black] (16.444412524211252,-3.678476680872202) node {1};
\draw [fill=wwqqzz] (9.593797902549547,-4.030313081547865) circle (2.5pt);
\draw[color=wwqqzz] (9.730581906045874,-3.5) node {A};
\draw [fill=qqffff] (10.474038090947696,-4.030313081547865) circle (2.5pt);
\draw[color=qqffff] (10.60300057394307,-3.5) node {B};
\draw [fill=ffqqqq] (11.442302298185659,-4.030313081547866) circle (2.5pt);
\draw[color=ffqqqq] (11.570247357916047,-3.5) node {C};
\draw [fill=qqqqff] (12.439907845036895,-4.030313081547866) circle (2.5pt);
\draw[color=qqqqff] (12.575425388319339,-3.5) node {D};
\draw [fill=zzwwqq] (13.320148033435045,-4.030313081547865) circle (2.5pt);
\draw[color=zzwwqq] (13.447844056216534,-3.5) node {E};
\draw [fill=qqffqq] (14.376436259512824,-4.030313081547866) circle (2.5pt);
\draw[color=qqffqq] (14.509918956265293,-3.5) node {F};
\end{scriptsize}
\end{tikzpicture}
		\caption{An example of a $y$ in $\D$ which cannot be approximated by injective curves. As part of our explanation we depict the image of the fixed points $A,B,C,D,E,F$ in $y_k$ and in the limit map $y$.}\label{devil}
	\end{figure}

\end{example}

To close the discussion on Problem \ref{prob-char} let us point out that the study of the intersection theory in algebraic topology seems to be very related. Indeed, the concept of \emph{weakly simple curves}, i.e., those for which, for each $\varepsilon >0$, there exists a simple curve whose Frechet distance to the studied one is at most $\varepsilon$, is closely related to non-interpenetration. A characterization of weakly simple curves, even in the piecewise affine setting, is highly desirable and also the only possibilities available up to know seem to be algorithmic \cite{algebra}.

{\color{black} Let us remark that even though the results of this work apply only  to the case of rods in the plane (that is deformations in $W^{1,p}((0,1); \mathbb{R}^2)$) it seems plausible that our proposed definition of non-interpenetration translates in verbatim also to the case of thin films in space.}

 {\color{black} Finally, we note that the setting considered in this work does not penalize rapid oscillations upon relaxation. Therefore, it cannot be avoided (by any growth of energy as $\mathrm{det}F \to 0$) that the null function lies in the set of }acceptable deformations for the rod. This paradox could most probably be only prevented by including higher gradients of deformation in the energy or by using a different scaling regime. Yet, it seems to be unknown if an equivalent of Theorem \ref{NotParticullarlyDeep} can hold in a smoother setting, i.e. we have the following:

\begin{Problem}
 Does it hold that the weak and strong closure of homeomorphisms are the same in $W^{k,p}((0,1); \mathbb{R}^2)$ with $k \geq 2$?   
\end{Problem}

\section{Proof of Theorem \ref{thm-dimRed}}

In this section, we give the detailed proof of Theorem \ref{thm-dimRed}. It requires a few auxiliary results. In the first one, we show that convexity is a crucial property for sequential weak lower semicontinuity of integral functionals along sequences of injective maps. 

\begin{lemma}\label{lemma;convexity-wlsc}
Suppose that $f: \mathbb{R}^2 \to \mathbb{R} \cup\{+\infty\}$ is finite everywhere except perhaps at zero. Then the integral functional 
 $$
 I(y) = \int_0^1 f(y'(x)) \mathrm{d} x
 $$
 with $I: W^{1,p}((0,1); \mathbb{R}^2) \to \mathbb{R}$ is weakly lower semicontinuous along sequences of homeomorphisms if and only if $f$ is convex on $\mathbb{R}^2$, i.e., finite everywhere.
\end{lemma}

\begin{proof}
One implication of the proof is automatic; it is well-known that the convexity of $f$ ensures weak lower semicontinuity even in larger sets than those of homeomorphisms.

Let us thus turn to proving necessity of the convexity. Let us take two vectors $a, b \in \mathbb{R}^2$ and $t \in (0,1)$.

Let us suppose at first that $a$ and $b$ are not co-linear. Then for every $n$, we may define the piecewise affine function $y_n$ through $y_n(0)=0$ and further inductively for $k=1\ldots n$ via
$$
y_n(x) = \begin{cases}
            y_n(\frac{k-1}{n})+a (x-\frac{k-1}{n}) & \text{ for } x \in (\frac{k-1}{n}, \frac{k-1}{n}+\frac{t}{n}] \\
            y_n(\frac{k-1}{n}+\frac{t}{n})+b (x-\frac{k-1}{n} -\frac{t}{n}) & \text{ for } x \in (\frac{k-1}{n}+\frac{t}{n}, \frac{k}{n}] 
\end{cases}
$$
As defined, the functions $y_n$ are continuous, piecewise affine, and injective.
Moreover, $y_n \rightharpoonup y$ where $y$ is the affine function $y=\big(ta+(1-t)b\big)x$. Thus, if $I$ is weakly lower semicontinuous along the sequence $
\{y_n\}$, we have that
\begin{align}\label{necessity-of-convexity}
\nonumber
 \liminf_{n \to \infty } I(y_n) &= \liminf_{n \to \infty }\sum_{k = 1}^n \int_{\frac{k-1}{n}}^{\frac{k-1}{n}+\frac{t}{n}} f(a) \mathrm{d} x+  \int_{\frac{k-1}{n}+\frac{t}{n}}^{\frac{k}{n}} f(b)  \mathrm{d} x \\ &=  t f(a) + (1-t) f(b)  \geq  f(ta+(1-t)b),   
\end{align}
the convexity condition for $a,b$ not co-linear.

It is well-known that a convex function is continuous \rone in the interior of its effective domain (see e.g. \cite{dacorogna}). However, even the weakened condition in \eqref{necessity-of-convexity} assures that $f$ is continuous in all points except for zero. We give the proof of this statement here for self-containment, even if it follows the lines of the standard argument.  

Take some point $0 \neq a \in \mathbb{R}^2$ and let us assume that $f$ is not continuous at $a$ so that there exists a sequence $\{a_n\}_{n \in \mathbb{N}}$ such that $a_n \to a$ but $f(a_k) \nrightarrow f(a)$. As $a \neq 0$ at least one of the components of $a$ needs to be non-zero. Let us assume that it is the first one. Without any loss of generality (and upon passing to a subsequence if necessary) we may assume that the sequences of the components, denoted as  $\{a^1\}_n$ and $\{a^2\}_n$, are monotone and at least one being strictly monotone. Here, and in the following, we will always indicate the \emph{component} by the upper index to avoid confusion with the sequence index.

Further, we may ask that exactly one of the following conditions holds
\begin{equation}
 f(a_n) - f(a) \geq \varepsilon \quad \text{or} \quad f(a) - f(a_n) \geq \varepsilon \quad \forall n\in \mathbb{N}.   
 \label{contradic}
\end{equation}
for some $\varepsilon>0$

Let us suppose that the first condition is met.  We take some fixed element $\tilde{a} \subset \{a\}_n$  {\color{black} and suppose that} $\tilde{a}^1 \neq 0$ and $\tilde{a}^1 \neq a^1$ and decompose elements $a_n$ (for $n$ big enough) as
$$
a_n = r_n \begin{pmatrix}
           \tilde{a}^1 \\
           a_n^{2} \\
         \end{pmatrix} + (1-r_n)\begin{pmatrix}
           a^1 \\
           a_n^{2} \\
         \end{pmatrix}
$$
with $r_n \in [0,1]$ which has to satisfy that $r_n \to 0$ as $n\to \infty$. {\color{black}
If no such element can found, it means that for $n$ big enough $a^1 =a^1_n$ so that we instead rewrite 
$$
a_n = t_1 \begin{pmatrix}
           a^1 \\
           \tilde{a}^2\\
         \end{pmatrix} + (1-t_1)\begin{pmatrix}
           a^1 \\
          a^2 \\
         \end{pmatrix},
$$
which makes the proof even easier.
}As the above convex combinations contains two vectors that are not co-linear, we may use the above convexity property to show that
$$
f(a_n) \leq r_n f \left( \begin{pmatrix}
           \tilde{a}^1 \\
           a_n^{2} \\
         \end{pmatrix}\right) + (1-r_n) f\left( \begin{pmatrix}
           a^1 \\
           a_n^{2} \\
         \end{pmatrix}\right)
$$
Decomposing further (if $a^2 \neq \tilde{a}^2$, otherwise this step is not necessary)
$$
\begin{pmatrix}
           a^1 \\
           a_n^{2} \\
         \end{pmatrix} = s_n \begin{pmatrix}
           a^1 \\
           \tilde{a}^2 \\
         \end{pmatrix} + (1-s_n)\begin{pmatrix}
           a^1 \\
           a^2 \\
         \end{pmatrix}
         \text{ and }
\begin{pmatrix}
           \tilde{a}^1 \\
           a_n^{2} \\
         \end{pmatrix} = t_n \begin{pmatrix}
           \tilde{a}^1 \\
           \tilde{a}^2 \\
         \end{pmatrix} + (1-t_n)\begin{pmatrix}
           \tilde{a}^1 \\
           a^2 \\
         \end{pmatrix}         
$$
with the numbers $s_n, t_n \in [0,1]$ converging to 0 as $n \to \infty$. Thus, we get 
\begin{align*}
f(a_n) &\leq r_n f \begin{pmatrix}
           \tilde{a}^1 \\
           a_n^{2} \\
         \end{pmatrix}+ (1-r_n) f \begin{pmatrix}
           a^1 \\
           a_n^{2} \\
         \end{pmatrix} \\
         &\leq r_n \left( s_n f \begin{pmatrix}
           \tilde{a}^1 \\
           \tilde{a}^2 \\
         \end{pmatrix}  + (1-s_n) f  \begin{pmatrix}
           \tilde{a}^1 \\
           a^2\\
         \end{pmatrix}\right) + (1-r_n) \left( t_n f \begin{pmatrix}
           a^1 \\
          \tilde{a}^{2} \\
         \end{pmatrix} +(1-t_n) f\begin{pmatrix}
           a^1 \\
          a^2 \\
         \end{pmatrix} \right) \\
         &\leq f(a) + C \max\{r_n,s_n,t_n\},
\end{align*}
where $C$ is a fixed number given by the values of the function $f$ in the fixed vectors constructed as above. Combining with \eqref{contradic} we see that
$$
f(a)+\varepsilon \leq f(a_n) \leq f(a) + C  \max\{r_n,s_n,t_n\},
$$
which yields a contradiction for $n$ large enough because $\max\{r_n,s_n,t_n\} \to 0$ as $n \to \infty$.

If the second condition in \eqref{contradic} holds, we show the claim in the case when $a^1, a^2 > 0$ and both sequences $\{a^1\}_n$ and $\{a^2\}_n$ a strictly monotone decreasing to $a^1$ and $a^2$ with modification to the other cases being straightforward. Then we decompose
$$
a = t_n \begin{pmatrix}
           {\color{black} (1/2)} a^1 \\
           a^2 \\
         \end{pmatrix} + (1-t_n)\begin{pmatrix}
           a_n^1 \\
           a^2 \\
         \end{pmatrix} \text{ and } \begin{pmatrix}
           a_n^1 \\
           a^2 \\
         \end{pmatrix} = r_n \begin{pmatrix}
           a_n^1 \\
           {\color{black} (1/2)} a^2 \\
         \end{pmatrix} + (1-r_n) \begin{pmatrix}
           a_n^1 \\
           a_n^2 \\
         \end{pmatrix}
$$
Now, we analogously to above use again the convexity property to show that 
$$
f(a_n)+\varepsilon \leq f(a) \leq f(a_n) + C \max\{r_n,t_n\},
$$
which is again a contradiction for $n$ large enough.

Having the continuity property at our disposal we may show the convexity inequality even for co-linear non-zero vectors $a$ and $b$ by defining 
$$
a_n = \begin{pmatrix}
           a^1+1/n \\
           a^2 \\
         \end{pmatrix}
$$
and taking the sequence 
$$
y_n(x) = \begin{cases}
            y_n(\frac{k-1}{n})+a_n (x-\frac{k-1}{n}) & \text{ for } x \in (\frac{k-1}{n}, \frac{k-1}{n}+\frac{t}{n}] \\
            y_n(\frac{k-1}{n}+\frac{t}{n})+b (x-\frac{k-1}{n} -\frac{t}{n}) & \text{ for } x \in (\frac{k-1}{n}+\frac{t}{n}, \frac{k}{n}] .
\end{cases}
$$
Now, all $y_n$ are piecewise affine and injective and converge weakly to the affine function $y(x) = (ta + (1-t)b)x$, which now may even be the zero function. Using the definition of weak lower semicontinuity again shows that 
$$
 \liminf_{n\to\infty} t f(a_n) + (1-t) f(b)  \geq  f(ta+(1-t)b)
$$
and thus relying on the continuity of $f$ in $a$, the claim. At this point, we already know that $f$ is finite in 0 and analogously to above, we show that it is continuous there, too. Now, we may again iterate and show the convexity property for all vectors when possibly replacing the zero vector by $ \begin{pmatrix}
           1/n \\
           0 \\
         \end{pmatrix}$ in the construction.

\end{proof}

As, even on the set of homeomorphisms, weak-lower semicontinuity of an integral functional on $W^{1,p}((0,1); \mathbb{R}^2)$ translates to the convexity of its density, it is natural to expect that the convex hull will also be in the relaxation of functionals taking infinite values outside the set of homeomorphisms. This is formalized as follows

\begin{lemma}[Relaxation]
\label{lemma-relax}
Let $p \in (1, \infty)$. Suppose that $f: \mathbb{R}^2 \to \mathbb{R}\cup\{+\infty\}$, $f \geq 0$, is finite and continuous everywhere except at zero, where it is infinite.  
Moreover, suppose that for every $\delta>0$, there exists $c_\delta>0$ such that it holds for all $a \in \mathbb{R}^2$, $|a|\geqslant \delta$ that
$$
f(a) \leqslant c_\delta\left(1+|a|^p\right).
$$
Then the relaxation of the integral functional
$$
\rone\mathcal{I}\EEE(y)= \begin{cases} 
\int_0^1 f (y'(x)) \mathrm{d} {x} &\quad \text{ if } {y} \text{ is a homeomorphism}, \\
+\infty &\quad \text{ else, }
\end{cases}
$$
with respect to the weak topology in $W^{1,p}((0,1); \mathbb{R}^2)$, that is the functional
\begin{equation}
\rone\mathcal{I}\EEE^\text{rel}(y) := \inf \{ \liminf_{k\to \infty} \rone\mathcal{I}\EEE(y_k); \{y_k\}_{k \in \mathbb{N}} \subset W^{1,p}((0,1),\mathbb{R}^2), y_k \rightharpoonup y \text{ in } W^{1,p}((0,1); \mathbb{R}^2)  \}
\label{wlsc-env}
\end{equation}
is given by
$$
\rone\mathcal{I}\EEE^\text{C}(y) = 
\begin{cases}
\int_0^1 \mathcal{C}f (y'(x)) \mathrm{d} {x} &\quad \text{ if } {y} \text{ is non-interpenetrative}, \\
+\infty &\quad \text{ else, }
\end{cases}
$$
where $Cf$ is the convex envelope of $f$, which is defined as 
$$\mathcal{C}f=\sup\{g:\mathbb{R}^2\to\mathbb{R} \textrm{ convex}: g\leq f\}.$$
Moreover, the sequence of homeomorphisms realizing the infimum in \eqref{wlsc-env} can be chosen as a piecewise affine.
    
\end{lemma}

\begin{remark}
Notice that the relaxation corresponds to the weakly lower semicontinuous envelope of the functional $\rone\mathcal{I}\EEE$. It is well known that a relaxation is a special case of the $\Gamma$-limit with a constant sequence of functionals \cite{braides}. Thus, as in the general case, we know that the relaxed functional has a minimizer which can be reached by infimizing sequences of the original functional.
\end{remark}

\begin{remark}[Attaining the convex envelope]
\label{lemma-convOsc}
Let us remind the reader that, since have that $f:\RR^2\to\mathbb{R}$ is continuous away from zero ($f(0)= +\infty$, and that $\lim_{|a|\to\infty}f(a)=+\infty$) a combination of three vectors is sufficient to calculate the convex envelope up to a small error of size $\gamma$. To be more precise, for a given $\xi\in\RR^2$, and $\gamma>0$ there are some $t^\gamma _i\ge 0$ with $\sum_{i=1}^3 t^\gamma_i=1$, and $a_i^\gamma \in\RR^2$ such that $\xi=\sum_{i=1}^3 t_ia_i^\gamma$,  and  the convex envelope $Cf$ evaluated at $\xi$  satisfies  (see \cite[Thm.~2.35]{dacorogna})
    $$Cf(\xi)+\gamma>\sum_{i=1}^3 t_i^\gamma f(a_i^\gamma )\ .$$
\end{remark}

\begin{proof}[Proof of Lemma \ref{lemma-relax}]
Let us choose $y \in W^{1,p}((0,1); \mathbb{R}^2)$ and an arbitrary sequence $\{y_k\}_{k \in \mathbb{N}} \subset W^{1,p}((0,1);  \mathbb{R}^2)$ with $y_k \rightharpoonup y$. Let us distinguish two cases: first, assume that $y$ is interpenetrative. Then, according to Theorem \ref{NotParticullarlyDeep}, there cannot exist a sequence of homeomorphisms converging weakly to $y$, for otherwise, $y$ would be non-interpenetrative. So, in this case 
$$
\rone\mathcal{I}\EEE^C(y) = \rone\mathcal{I}\EEE^\text{rel}(y) = +\infty.
$$

Let us thus concentrate on the case when $y$ is non-interpenetrative and $\{y_k\}$ is a sequence of homeomorphisms \rone converging weakly to $y$. \EEE  We then have 
$$
\liminf_{k \to \infty} \rone\mathcal{I}\EEE(y_k)  =  \liminf_{k \to \infty} \int_0^1 f(y_k')\, \mathrm{d}x \geq \liminf_{k \to \infty} \int_0^1 Cf(y_k')\, \mathrm{d}x  \geq \rone\mathcal{I}\EEE^C(y'),
$$
and since this is independent of the particular sequence, we obtain $\rone\mathcal{I}\EEE^\text{rel}(y) \geq \rone\mathcal{I}\EEE^C(y) $.

Now, we need to show the opposite inequality. To this end, we construct a sequence {\color{black} $\{\bar{y}_k^{n(k)}\}_{k \in \mathbb{N}} \subset W^{1,p}((0,1); \mathbb{R}^2)$ such that
$$
\rone\mathcal{I}\EEE^C(y) \geq \lim_{k \to \infty} \mathcal{I}(\bar{y}_k^{n(k)}).
$$}
First, as $y$ is non-interpenetrative,  we may appeal to Theorem \ref{NotParticullarlyDeep} and choose a \emph{piece-wise affine} approximating (in the strong topology) sequence $\{\bar{y}_k\} \subset W^{1,p}((0,1); \mathbb{R}^2)$ such that all $\bar{y}_k$ are homeomorphisms. Now, because 
{\color{black}
$$
\{\mathcal{C}f < +\infty\} = \mathrm{co}(\{f<+\infty\})=\mathrm{co}(\mathbb{R}^2\setminus\{0\}) = \mathbb{R}^2,
$$
where $\mathrm{co}(M)$ is the convex hull of the set $M \subset \mathbb{R}^2$, we know that $\mathcal{C}f$ is finite on $\mathbb{R}^2$ and thus continuous.
Therefore, }we can fix $\delta > 0$ and find a constant $\tilde{c}$ such that $\mathcal{C}f(a) \leq \tilde{c}$ if $|a| \leq \delta$. Moreover, for $|a| \geq \delta$ the function $\mathcal{C}f$ inherits the growth condition of $f$ so that in sum
$$
\mathcal{C}f(a) \leq c(1+|a|^p) \qquad \text{ on } \mathbb{R}^2. 
$$
Thus, we may use the continuity of Nemytskii mappings (see e.g. \cite{roubicek})  to assert that 
$$
\lim_{k \to \infty} \rone\mathcal{I}\EEE^C(\bar{y}_k) = \rone\mathcal{I}\EEE^C(y),
$$
so that when choosing a subsequence of $k$'s if necessary, we have that
$$
\rone\mathcal{I}\EEE^C(y) \geq \rone\mathcal{I}\EEE^C(\bar{y}_k)- \frac{1}{k}.
$$

Let us now fix $k$ and one $\bar{y}_k$. We can assume $\bar{y}_k$ is affine on the intervals $\left[x_k^l,x_k^{l+1}\right)$, $l=0\ldots N(k)$. We may find numbers $\beta(k)$, $n(k)$ such that 
$$
|x_k^{l+1}-x_k^l| {\color{black} >} \beta(k)  \text{ for all } l=0\ldots N(k),
$$
and, moreover, for all $x \in \left(x_k^{l}+\beta(k) ,x_k^{l+1}-\beta(k)\right)$ with $l=0\ldots N(k)$ we have that 
$$
B\Big(\bar{y}_k(x),\frac{1}{n(k)}\Big)\cap \bar{y}_k(\tilde{x}) = \emptyset \text{ for all } \tilde{x} \in (x_k^\ell,x_k^{\ell+1}) \text{ with } l\neq \ell.
$$
In other words, as long as we are at a $\beta(k)$-distance from the points $x_k^l$ we can find a neighborhood around the piecewise affine function, where we can perturb it without interfering with injectivity. That is exactly what we will do and we will construct a function {\color{black} $\bar{y}_k^{n(k)}$} by replacing $\bar{y}_k$ by 
suitable oscillations in order to attain the convex envelope. 

Let us detail the particular construction on one interval $(x_k^l, x_k^{l+1})$, where we denote $\xi := \bar{y}_k'$. Then we know from Remark \ref{lemma-convOsc} that for a given $\gamma>0$ there are some $t_{\gamma,i}\ge 0$ and $a_{\gamma,i} \in\RR^2$ satisfying $\sum_{i=1}^3 t_{\gamma,i}=1$ and $\xi=\sum_{i=1}^3 t_{\gamma,i}a_{\gamma,i}$  such that
    $$\mathcal{C}f(\xi)+\gamma>\sum_{i=1}^3 t_{\gamma,i} f(a_{\gamma,i})\ .$$
Let us define for $x\in[0,1]$ the auxiliary function
 $$
    y^\gamma(x)=\begin{cases}
        a_{\gamma,1} x & \text{ if } x\in[0,t_{\gamma,1}],\\
        a_{\gamma,2}(x-t_{\gamma,1})+a_{\gamma,1} t_{\gamma,1} &\text{ if } x\in[t_{\gamma,1},t_{\gamma,1}+t_{\gamma,2}],\\
       a_{\gamma,3}(x-t_{\gamma,1}-t_{\gamma,2})+a_{\gamma,1} t_{\gamma,1}+a_{\gamma,3} t_{\gamma,2} & \text{ if } x\in[t_{\gamma,1}+t_{\gamma,2},1],
    \end{cases}
    $$

Now, as defined, $y^\gamma$ might not be a homeomorphism at this point. This could happen if either one of the vectors $a_{\gamma,i}$ is the zero vector, or if two (or all three) of the vectors are mutually co-linear. As $f$ is infinite at 0, we may exclude the first case, but the second might still happen. If so, we may assume, without loss of generality, that $a_{\gamma,1}$ and $a_{\gamma,2}$ are co-linear and $a_{\gamma,3}$ is either non-colinear with either or aligned with $a_{\gamma,2}$ (which means that \rone  $a_{\gamma,3} = \alpha a_{\gamma,2}$ \EEE with $\alpha >0$). As $a_{\gamma,1} \neq 0$ we can find one component of $a_{\gamma,1}$ which is non-zero and without loss of generality suppose it is the first one. Then we can use the alternative definition
     $$
    y^{\gamma}(x)=\begin{cases}
        \begin{pmatrix}
           a_{\gamma,1}^1+a_{\gamma,1}^1/2 \\
           a_{\gamma,1}^2 \\
         \end{pmatrix}x & \text{ if } x\in[0,t_{\gamma,1}/2],\\  
        \begin{pmatrix}
           a_{\gamma,1}^1+a_{\gamma,1}^1/2 \\
            a_{\gamma,1}^2 \\
         \end{pmatrix}\frac{t_{\gamma,1}}{2} + \begin{pmatrix}
           a_{\gamma,1}^1-a_{\gamma,1}^1/2 \\
           a_{\gamma,1}^2 \\
         \end{pmatrix}(x-\frac{t_{\gamma,1}}{2}) & \text{ if } x\in[t_{\gamma,1}/2,t_{\gamma,1}] \\
        a_{\gamma,2}(x-t_{\gamma,1})+a_{\gamma,1} t_{\gamma,1} &\text{ if } x\in[t_{\gamma,1},t_{\gamma,1}+t_{\gamma,2}],\\
        a_{\gamma,3}(x-t_{\gamma,1}-t_{\gamma,2})+a_{\gamma,1} t_{\gamma,1} +a_{\gamma,2}t_{\gamma,2} & \text{ if } x\in[t_{\gamma,1}+t_{\gamma,2}{\color{black}, 1}],
    \end{cases}
    $$
    where we, as before, denoted vector components by upper indices.
    
    Now, $ y^\gamma$ is again a homeomorphism that equals to ${\color{black} \xi {\cdot} x}$ at the boundary of $(0,1)$.

    We now scale the found functions $y^\gamma$. {\color{black} For  $n\in\N$, we set 
    $$y^\gamma_n(x)=\frac{1}{n} y^\gamma(nx)$$  
    if $x \in [0, 1/n]$ and
    \begin{align*}
     y^\gamma_n(x)& =y^\gamma_n(x-(i-1)/n) +y^\gamma_n((i-1)/n), 
    \end{align*}
 if $x \in [(i-1)/n, i/n]$, for $2 \leq i \leq n$.}   
Then   $y^\gamma_n \rightharpoonup \xi\cdot x$ and {\color{black} $y^\gamma_n(0)=0$ and $y^\gamma_n(1)=\xi$}.
    
    Now, a calculation analogous to \eqref{necessity-of-convexity}, the fact that $\mathcal{C}f\le f$, and Lemma~\ref{lemma;convexity-wlsc}  show that 
    $$\mathcal{C}f(\xi)=\int_0^1\mathcal{C}f(\xi)\,\mathrm{d}x\le \lim_{n\to\infty}\int_0^1 f((y^\gamma_n)')\,\mathrm{d}x<\int_0^1\mathcal{C}f(\xi)\,\mathrm{d}x+\gamma=\mathcal{C}f(\xi)+\gamma .$$
    As $\gamma>0$ was arbitrary 
     a uniform bound on $(y^\gamma_n)'$ implied by coercivity of $f$ allows us to extract a diagonal sequence for {\color{black} $\gamma(n)\to 0$ and $n\to\infty$ } such that  
    $$\mathcal{C}f(\xi)=\int_0^1\mathcal{C}f(\xi)\,\mathrm{d}x= \lim_{n\to\infty}\int_0^1 f((y^{\gamma(n)}_n)')\,\mathrm{d}x,$$
    as well as 
    $$\mathcal{C}f(\xi)\geq \int_0^1 f((y^{\gamma(n)}_n)')\,\mathrm{d}x - \frac{1}{n}.$$
    
We now set 
$$
\bar{y}_k^n(x) = \begin{cases}
    \bar{y}_k(x) & \text{ if } x \in  
    [x_k^{l}, x_k^l+\beta(k)] \text{ or } [x_k^{l+1}-\beta(k), x_k^{l+1})
    \\
    \bar{y}_k(x_k^l  +\beta(k)) \\ \quad+(x_k^{l+1}-x_k^l-2\beta(k))y^{\gamma(n)}_n\big(\frac{x-x_k^l-\beta(k)}{x_k^{l+1}-x_k^l-2\beta(k)} \big)  &\text{ if } x \in  (x_k^l  +\beta(k), x_k^{l+1}-\beta(k))
\end{cases}
$$
and by a change of variables we obtain that
\begin{align*}
\int_{x_k^l  +\beta(k)}^{x_k^{l+1}-\beta(k)} \mathcal{C}f(\bar{y}_k') \mathrm{d}x &= \int_{x_k^l  +\beta(k)}^{x_k^{l+1}-\beta(k)} \mathcal{C}f(\xi) \mathrm{d}x = (x_k^{l+1}-x_k^l-2\beta(k)) \mathcal{C}f(\xi) \\&\geq (x_k^{l+1}-x_k^l-2\beta(k)) \int_0^1 f((y^{\gamma(n)}_n)')\,\mathrm{d}x - \frac{(x_k^{l+1}-x_k^l-2\beta(k))}{n} \\ &= \int_{x_k^l  +\beta(k)}^{x_k^{l+1}-\beta(k)} f([\bar{y}_k^n]') \mathrm{d} x -  \frac{(x_k^{l+1}-x_k^l-2\beta(k))}{n}.
\end{align*}
\rone Then  $\bar{y}_k^{n(k)}$ with $n(k)\to\infty$ for $k\to\infty$  is a   sought recovery sequence because \EEE
\begin{align*}
\rone\mathcal{I}\EEE^C(y) &\geq \rone\mathcal{I}\EEE^C(\bar{y}_k)- \frac{1}{k} =
\sum_{l=0}^{N(k)} \int_{x_k^l}^{x_k^{l+1}} \mathcal{C}f(\bar{y}_k') \mathrm{d}x  - \frac{1}{k} \geq \sum_{l=0}^{N(k)} \int_{x_k^l+\beta(k)}^{x_k^{l+1}-\beta(k)}   \mathcal{C}f(\bar{y}_k') \mathrm{d}x -  \frac{1}{k} \\
& \geq \sum_{l=0}^{N(k)} \int_{x_k^l+\beta(k)}^{x_k^{l+1}-\beta(k)}   f([\bar{y}_k^{n(k)}]') \mathrm{d} x - \frac{1-2\beta(k) N(k)}{n(k)}  - \frac{1}{k} \\ &\geq   \sum_{l=0}^{N(k)} \bigg(\int_{x_k^l}^{x_k^{l+1}} f([\bar{y}_k^{n(k)}]')\mathrm{d} x - \beta(k) f({\color{black} \bar{y}_k}'(x_k^l))- \beta(k) f({\color{black} \bar{y}_k}'(x_k^{l+1})) \bigg)  - \frac{1-2\beta(k) N(k)}{n}  - \frac{1}{k}\\
&\geq   \sum_{l=0}^{N(k)} \bigg(\int_{x_k^l}^{x_k^{l+1}} f([\bar{y}_k^{n(k)}]')\mathrm{d} x - 2\beta(k) \int_{x_k^l}^{x_k^{l+1}} f({\color{black} \bar{y}_k}'(x)) \mathrm{d} x \bigg)  - \frac{1-2\beta(k) N(k)}{n(k)}  - \frac{1}{k} \\
&\geq \rone\mathcal{I}\EEE(\bar{y}_k^{n(k)}) - {\color{black}2}\beta(k) \rone\mathcal{I}\EEE({\color{black} \bar{y}_k})- \frac{1}{n(k)}  - \frac{1}{k},
\end{align*}
{\color{black} where in the last line we omitted the the positive term $2\beta(k) N(k)/n$.}
Now letting $k$ to infinity, $\rone\mathcal{I}\EEE({\color{black} \bar{y}_k})$ may diverge but as we have the freedom to choose $\beta(k)$ small enough, we can still assure that  $\beta(k) \rone\mathcal{I}\EEE({\color{black} \bar{y}_k})$ converges to zero, which proves the claim.
\end{proof}

\begin{lemma}[Tubular lemma]
\label{tubular-lemma-1}
Let $y \in W^{1,p}((0,1); \mathbb{R}^2)$ be injective and Lipschitz. Let   $b \in C^1((0,1);\mathbb{R}^2)$ be such that $\mathrm{det}(y'|b) \geq \delta$. Then, we can find a number $h = h(L,\delta, \|b\|_{C^1((0,1);\mathbb{R}^2)})$ such that 
\begin{align*}
\rone y_h: \EEE (0,1) \times [-h/2,h/2] \to \mathbb{R}^2 \ni (x^1,x^2) \mapsto y(x^1) +  x^2 b(x^1) 
\end{align*} 
is a homeomorphism.
 \end{lemma}
 
 \begin{proof}
\rone From its definition, the function $y_h$ is Lipschitz and we may find $h_0 > 0$ such that 
 $$\mathrm{det}(\nabla y_h(x^1,x^2)) = \mathrm{det}(y'+x^2b'|b) \geq \delta/2$$ 
 for all $h \leq h_0$. \EEE

 Moreover, we can find find $h_1 \leq h_0$ (if necessary) and a number $\gamma$ such that for all $h \leq h_1$ and all $\xi \in (0,1)$ we have that 
 $$
 |y_h(x_1)-y_h(x_2)| \geq \frac{\delta}{4} |x_1-x_2| \quad \text{ whenever } |x_1^1-\xi| \leq \gamma \text{ and } |x_2^1-\xi| \leq \gamma.
 $$
Indeed, taking $x_1$ and $x_2$ in $(0,1) \times [-h/2,h/2]$, and letting $\varphi(t) = x_1 + t(x_2-x_1)$, we may write
\begin{align*}
 |y_h(x_1)-y_h(x_2)| &= \Big|\int_0^1 \nabla y_h(\varphi(s)) \varphi'(s) \mathrm{d} s \Big | 
 \\ &=
 \Big|\int_0^1 \nabla y_h(\xi,0) (x_2-x_1)  + \nabla \big(y_h(\varphi(s))-y_h(\xi,0)\big) (x_2-x_1) \mathrm{d} s \Big | \\
 &\geq \bigg(\frac{\delta}{2}  - {\color{black}\|\nabla y_h\|_{L^\infty((0,1) \times [-h/2,h/2]; \mathbb{R}^2)}} \sqrt{h_1^2+4\gamma^2}\bigg) |x_2-x_1| \\
 &\geq \frac{\delta}{4} |x_2-x_1|,
\end{align*}
where the last inequality is by choosing $h_1$ and $\gamma$ small enough. 

Particularly, we have a kind of uniform local injectivity implying that whenever $x_1, x_2$ are such that $|x_1^1-x_2^1| \leq 2 \gamma$, $y_h(x_1) \neq y_h(x_2)$ for all $h \leq h_1$.

Finally, let us ensure that $y_h$ is also globally injective which will ensure the claim. To this end, we may need to choose $h$ once again smaller if necessary. Indeed, as $y$ is a \emph{globally} injective mapping, there has to be a number $\alpha > 0$ such that for any two $x_1,x_2$ with $|x_1^1-x_2^1| \geq 2 \gamma$ we have that $|y(x_1)-y(x_2)| \geq \alpha$.
Thus, choosing $h_3 \leq h_2$ and $h_3 < \frac{\alpha}{4\|b\|_{C^0((0,1);\mathbb{R}^2)}}$ yields that $y_h$ is a global homeomorphism for all $h \leq h_3$. 
\end{proof}

\begin{remark}[Smoothing and the positivity of the Jacobian]
\label{smoothing-det}
Let us assume that we have a piece-wise affine homeomorphism $y_k(x)$ such that \rone  $y_k = a_k^l$ \EEE on $(x_k^l, x_k^{l+1})$, $l=0 \ldots N(k)$, with $a_k^l$ some given vectors. Assume moreover that we are given the piece-wise constant function $b_k = \xi_{a_k^l}$ on $(x_k^l, x_k^{l+1})$, $l=0 \ldots N(k)$ such that 
$$
\mathrm{det}(a_k^l|\xi_{a_k^l}) \geq 2\delta,
$$
then we can find a smooth functions ${\bf b}_{k,i}$ that converge strongly (as $i \to \infty$) in $L^p((0,1); \mathbb{R}^2)$ to the piece-wise constant function $b_k$ and a number $\varepsilon = \varepsilon(y_k,\delta)$ 
\begin{equation}
\label{pos-det-rem}
\mathrm{det}(y_k'(x)|b_{k,i}) \geq \varepsilon. \qquad \text{ for a.a. $x \in (0,1)$},
\end{equation}
Let us first notice that applying the standard mollifier by convex averaging may fail in providing the right approximants. Indeed, take the situation as depicted in Figure \ref{neg-det-fig}. Then, applying the standard convex averaging will result in the vector $b_{k,i}$ corresponding to the dashed arrow somewhere near the corner, and thus, the orientation will change. 
\begin{figure}
    \centering
    \includegraphics[scale=0.6]{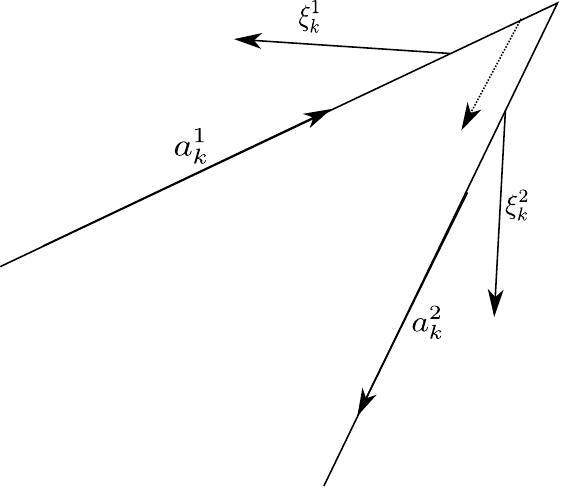}
    \caption{An illustration of how convex averaging might change the orientation.}
    \label{neg-det-fig}
\end{figure}
However, we can fix the smoothing process by introducing the following procedure: We introduce the intermediate vector $\zeta_k^l$ defined as 
$$
\zeta_k^l = 
\begin{cases}
    -a_k^l+a_k^{l+1} & \text{ if } \mathrm{det}(a_k^l|a_k^{l+1}) >0, \\
    a_k^l-a_k^{l+1} & \text{ if } \mathrm{det}(a_k^l|a_k^{l+1}) < 0 \\
    \xi_{a_k^l}  & \text{ if } \mathrm{det}(a_k^l|a_k^{l+1}) = 0
\end{cases}
$$
Now, we define the vector 
$$
b_{k,i}(x) = \begin{cases}
    \zeta_k^l & \text{ if } x \in (x_k^{l+1}-\frac{1}{i}, x_k^{l+1}+\frac{1}{i}), l=0\ldots N(k)-2 \\
    \xi_{a_k^l} & \text{ if } x \in (x_k^l+\frac{1}{i},x_k^{l+1}-\frac{1}{i}) \text{ for } l=1...N(k)-1 \\
    \xi_{a_k^0} &\text{  for } x \in (0,x_k^0-\frac{1}{i})  \\
    \xi_{a_k^{N(k)}} &\text{ for }  x \in (x_k^{N(k)}+\frac{1}{i},1),
\end{cases}
$$
so that basically we only alter the original piece-wise constant function $b_k$ near the corners. At this point, for the construction to make sense, we need to assume $i$ large enough (as compared to $\mathrm{min}_l |x_k^{l+1}-x_k^l|$) and eventually, as $i \to \infty$, we will have changed it only on a negligible set. 

This alternation has the purpose to now be able to apply the standard mollifier with the diameter of its kernel smaller than $\frac{1}{2i}$ because
$\det(a_k^l|\zeta_k^l) > 2\varepsilon$ and $\det(a_k^{l+1}|\zeta_k^l) > 2\varepsilon$. Indeed, if $\mathrm{det}(a_k^l|a_k^{l+1}) >0$ we have that 
\begin{align*}
\det(a_k^l|\zeta_k^l) &= \det(a_k^l|-a_k^l+a_k^{l+1}) = \det(a_k^l|a_k^{l+1}) \\
\det(a_k^{l+1})|\zeta_k^l) &= \det(a_k^{l+1})|-a_k^l+a_k^{l+1}) = \det(a_k^l|a_k^{l+1}),
\end{align*}
and the calculation is analogous in the case when $\mathrm{det}(a_k^l|a_k^{l+1}) <0$ while if this determinant is zero $a_k^l$ and $a_k^{l+1}$ are aligned and thus $\det(a_k^{l+1}|\xi_k^l)$ is also positive. As $k$ is finite, we can find a minimum value (that only depends on $y_k$ and $\delta$) so that we see that
$$
\det(a_k^l|\zeta_k^l) > \tilde{\varepsilon} \qquad \text{for all } l=0 \ldots k-1.
$$
After this modification, the convex averaging translates to the determinant because the determinant is multi-linear (and so convex averaging of the second column translates to convex averaging of the determinant) and thus
$$
\mathrm{det}(y_k'(x)|b_{k,i}) \geq \mathrm{min}(\tilde{\varepsilon} , \delta).
$$
\end{remark}

\begin{proof}[Proof of Theorem \ref{thm-dimRed}]  
To  prove the veracity of the first point of Definition~\ref{GammaConvDef}, let us take a sequence $\{y_h\}_{h>0} \subset W^{1,p}(\Omega, \mathbb{R}^2)$ such that $\pi(y_h) \to y$ in $L^p((0,1); \mathbb{R}^2)$. If $\liminf_{h\to 0 }J_h(y_h) = +\infty$, the statement is automatic. Thus, we may, without loss of generality restrict ourselves to a subsequence of $h$'s (not relabelled) such that $J_h(y_h) \leq C$. Thus all $y_h$'s are homeomorphisms on $\Omega$.

As the energy $W$ in \eqref{bulkFunct} satisfies the coercivity 
$$
W(F) \geq c |F|^p,
$$
we know that the sequence $\{y_h\}$ under question is actually bounded in $W^{1,p}(\Omega; \mathbb{R}^2)$ and thus, we may select a non-relabeled sub-sequence such that 
$$
y_h \rightharpoonup \tilde{y}  \text{ in } W^{1,p}(\Omega; \mathbb{R}^2)
$$
for some \rone $\tilde y\in W^{1,p}(\Omega; \mathbb{R}^2)$. \EEE
However, we can  even show a stronger statement: indeed, due to the scaling in the second variable, we have that
$$
\int_\Omega |\partial_1 y_h|^p+\frac{1}{h^p}|\partial_2 y_h|^p \mathrm{d}x \leq C,
$$
so that $\partial_2 y_h \to 0$ strongly in $L^p(\Omega; \mathbb{R}^2)$. This, in particular, means that $\tilde{y}$ is constant in the second variable so that we can identify 
$$
 \tilde{y}(x,0) = \lim_{h\to 0} \pi(y_h) = y(x).
$$

 In the case $p>2$, we know that the functions $\{y_h\}$ in the sequence are uniformly H\"older-continuous, and thus the same holds for the restrictions 
\rone $y_h^\mathrm{r}=y_h(\cdot, 0)$. \EEE In the case $p=2$,  we apply \cite[Corollary~7.5.1]{IM} to get a uniform modulus of continuity for the sequence and its limit \rone obtained by the Arzel\'{a}-Ascoli theorem. \EEE Therefore, the limiting function $y$  is non-interpenetrative by definition.

Now we have that
$$
    \begin{aligned}
    \liminf_{h \to 0} J_h(y_h) &= \liminf_{h \to 0} \int_\Omega W (\partial_1 y_h(x)| h^{-1} \partial_2 y_h(x) ) \mathrm{d} {x}\\
    &\geq  \int_0^1 \int_{-1/2}^{1/2} \liminf_{h \to 0} W_0(\partial_1 y_h(x)) \mathrm{d} x \geq J(y),
\end{aligned}
$$ 
where $W_0(a) = \mathcal{C}\min_{\xi \in \mathbb{R}^2} W(a| \xi)$, the first inequality is due to Fatou's lemma, the pointwise inequality $W(a|\zeta) \geq W_0(a)$ satisfied for all $\zeta \in \mathbb{R}^2$. The last inequality is by the uniform convergence mentioned above; recall here also that the the rod energy functional $J(y)$ is given in \eqref{rodFunct}.
      
Now, we need to construct the recovery sequence. To this end, let us fix $y \in W^{1,p}((0,1); \mathbb{R}^2)$. If $y$ is interpenetrative  $J(y) = +\infty$ so that any sequence approximating $y$ (e.g. as in standard dimension reduction arguments) would do as a recovery sequence.

Let us thus concentrate on the case when $y$ is non-interpenetrative. In this case, want to appeal to Lemma \ref{lemma-relax} and thus define the auxiliary function
\rone
\begin{equation}
\label{f-def}
f(a) := \min_{\xi\in\mathbb{R}^2} W(a|\xi).    
\end{equation} \EEE

Notice that, as $W$ is {\color{black} $p$-coercive} and continuous, the minimum can always be found. 
Notice also that, as for every {\color{black}$0 \neq a \in \mathbb{R}^2$}, the matrix
$$
\begin{pmatrix}
           a_1 & -a_2 \\
           a_2 & a_1 \\
         \end{pmatrix}
$$
has positive determinant, $f$ is finite everywhere in $\mathbb{R}^2$ except in $0$, where  it is infinite.
Then, owing to  Lemma \ref{lemma-relax}, we find a sequence of \emph{piece-wise} affine homeomorphisms $\{\bar{y}_k\} \subset W^{1,p}((0,1); \mathbb{R}^2)$ such that 
\begin{equation}
\label{eq-recov1}
 J(y) \geq \int_0^1 f(\bar{y}_k') \mathrm{d} x - \frac{1}{k}  .
\end{equation}
Let us now fix $k$ and one $\bar{y}_k$. We can assume $\bar{y}_k$ is affine on the intervals $(x_k^l,x_k^{l+1})$, $l=0\ldots N(k)$. Then, we can further continue rewriting \rone the right-hand side of\EEE \eqref{eq-recov1} as 
$$
\int_0^1 \min_{\xi\in \mathbb{R}^2} W(\bar{y}_k'|\xi) \mathrm{d} x - \frac{1}{k} = \sum_{l=0}^{N(k)} \int_{x_k^l}^{x_k^{l+1}}  \min_{\xi\in \mathbb{R}^2} W(a_k^l|\xi) \mathrm{d} x - \frac{1}{k},
$$
where $a_k^l = [\bar{y}_k]'$ on $(x_k^l,x_k^{l+1})$. Let us denote the minimizer in \eqref{f-def} as $\xi_{a_k^l}$ and take the piece-wise constant function $b_k = \xi_{a_k^l}$ on $(x_k^l,x_k^{l+1})$, $l=0\ldots N(k)$. Notice that for some constant $\delta=\delta(k)$, we have that
$$
{\color{black}\mathrm{det}(a_k^l|\xi_{a_k^l})} \geq 2\delta
$$
for all $l=0\ldots N(k)$.

We now apply Remark \ref{smoothing-det} to $b_k$ in order to define a smooth function $b_{k,i}$ ($i \in \mathbb{N})$ such that $\mathrm{det}([a_k^l]|b_{k,i}) \geq \varepsilon(k)$ for all $x \in (0,1)$. Moreover, as the $W$ is $p$-coercive, it is evident that $b_k \in L^p((0,1); \mathbb{R}^2)$ and thus we can assume that $b_{k,i} \to b_k$ strongly in $L^p((0,1); \mathbb{R}^2)$ (as  $i \to \infty$), whence, when appealing also to the $p$-growth of $W$ (with a constant $C_\varepsilon$) on the set where $\det(F) \geq \varepsilon$ and thus to continuity of the Nemytskii operator (see \cite[{\color{black} Sect. 1.3}]{roubicek}), we may choose $i \geq i_0(k)$ such that 
$$
\sum_{l=0}^{N(k)} \int_{x_k^l}^{x_k^{l+1}}  \min_{\xi\in \mathbb{R}^2} W([a_k^l]|\xi) \mathrm{d} x - \frac{1}{k} \geq \sum_{l=0}^{N(k)} \int_{x_k{k}^l}^{x_k^{l+1}}  W([\bar{y}_k]'|b_{k,i}) \mathrm{d} x - \frac{2}{k}
$$

At this point, we can use Lemma \ref{tubular-lemma-1} to define for $h \leq h_0(i,k)$ a homeomorphism $y_{h,k,i}(x_1,x_2) = \bar{y}_k(x_1) + x_2b_{k,i}$. Using the  $p$-growth of $W$ (with a constant $C_\delta$) on the set where $\det(F) \geq \delta$ and thus to Nemytskii continuity again, we may choose $h_0$ smaller if necessary to see that
$$
J(y) \geq \sum_{l=0}^{N(k)} \int_{x_k^l}^{x_k^{l+1}} W([\bar{y}_k]'|b_{k,i}) \mathrm{d} x - \frac{2}{k} \geq \frac{1}{h} \int_{-h/2}^{h/2} W(\nabla y_{h,k,i}) \mathrm{d} x -  \frac{3}{k},
$$
which yields the desired recovery sequence and thus the claim.
\end{proof}

\section{Proof of Theorem \ref{NotParticullarlyDeep}}

In this Section, we prove Theorem \ref{NotParticullarlyDeep}. To this end, introduce some additional notation to simplify the writing later on. We start by setting
$$
\I:=\bigl\{\phi:(0,1)\to\er^2:\  \phi \in W^{1,p}((0,1);\er^2) \text{ such that } \phi \text{ is continuous and one-to-one}\bigr\}
$$
and the weak and strong closures of the set
$$
\begin{aligned}
\Is&:=\bigl\{\phi:(0,1)\to\er^2:\ \exists \{\phi_k\}_{k \in \mathbb{N}}\subset \I \text{ such that }\phi_k\to \phi\text{ strongly in }W^{1,p}((0,1);\er^2)\bigr\}, \\
\Iw&:=\bigl\{\phi:(0,1)\to\er^2:\ \exists \{\phi_k\}_{k \in \mathbb{N}}\subset \I  \text{ such that }\phi_k\rightharpoonup \phi\text{ weakly in } W^{1,p}((0,1),\er^2)\bigr\}.\\
\end{aligned}
$$

With this notation, we reformulate Theorem \ref{NotParticullarlyDeep} as 

\vspace{1ex}

\noindent{\bf Theorem~\ref{NotParticullarlyDeep} (reformulated).}{\it 
	Let $p\in[1,\infty)$ and $\phi \in W^{1,p}((0,1); \er^2)$. The following conditions are equivalent
	\begin{enumerate}
		\item $\phi\in \Is$. 
		\item $\phi\in \Iw$. 
		\item There are $\tilde{\phi}_k\in\I$ that are equicontinuous and $\tilde{\phi}_k\sto \phi$ on $(0,1)$, where $\sto$ denotes the uniform convergence.
	\end{enumerate}   
Moreover, the closure $\Is$ stays the same even if we only admit piece-wise affine sequences.  
}

\vspace{1ex}

	\begin{remark}
		In Theorem~\ref{NotParticullarlyDeep}, we allow for the weak convergence $\phi_k \deb\phi$ in $W^{1,1}((0,1);\er^2)$  but we assume that $\phi\in W^{1,1}((0,1);\er^2)$ and hence we obtain the uniform integrability of $\phi_k'$. A similar characterization for $w^*$, strict or area-strict limits in $\mathrm{BV}((0,1);\er^2)$ would be more complicated as the limit class is not contained in $\mathcal{C}^0((0,1);\er^2)$ which we illustrate in Example~\ref{simple}. Continuity is a  natural assumption for us as discontinuous curves may pass ``through'' themselves quite often. 
	\end{remark}

\begin{example}\label{simple} We define the continuous piece-wise linear injective functions $h_k:(0,1) \to \er^2$ as follows. Let $h_k(0):=0$ and 
$$
h_k'(x):=
\begin{cases}
\frac{1}{k}&\text{ for }x\in[0,\tfrac{1}{2}-\tfrac{1}{k}]\cup[\tfrac{1}{2}+\tfrac{1}{k},1],\\
\frac{k}{2}&\text{ for }x\in[\tfrac{1}{2}-\tfrac{1}{k},\tfrac{1}{2}+\tfrac{1}{k}].\\
\end{cases}
$$
Then we define $\phi_k(x):=[h_k(x),0]$. This sequence is bounded in $W^{1,1}((0,1);\er^2)$ and so $\phi_k \debst \phi$ in $BV((0,1);\er^2)$ where 
$$
\phi(x):=
\begin{cases}
[0,0]&\text{ for }x\in[0,\tfrac{1}{2}],\\
[1,0]&\text{ for }x\in[\tfrac{1}{2},1].\\
\end{cases}
$$
but $\phi$ is not continuous.
\end{example}

The rest of this section will be devoted to proving Theorem \ref{NotParticullarlyDeep}, with the main difficulty stemming from the construction part. This means that, being given a function $\varphi \in W^{1,p}((0,1); \mathbb{R}^2)$ such that it can be approximated by homeomorphisms in the $C^0$-norm, we need to show that we can construct piece-wise affine injective approximations of $\varphi$ in the strong topology of $W^{1,p}((0,1); \mathbb{R}^2)$. Before embarking on the technical details of the proof of Theorem \ref{NotParticullarlyDeep}, we give here (heuristically) the main ideas of the proof.

First, we introduce a grid, as is common in constructing piece-wise affine approximations. Importantly, however, the grid is chosen in the \emph{image}, so in a sense we partition $\varphi((0,1))$. This gives us better control over the possible non-injectivities. Also, we will need to choose the grid carefully (and call it then a good arrival grid, see Definition \ref{good arrival grid}), so that, near each crossing point of $\varphi((0,1))$ and the grid, the function $\varphi$ behaves essentially like an affine function that passes from one rectangle of the grid to another. See Figure \ref{grid}, where one possible grid is shown. Also, we refer to Figure \ref{goodgrid}, where the choice of the good arrival grid is illustrated in more detail.

\begin{figure}[h t p]
{\begin{center}
     {\includegraphics[width=0.9\textwidth]{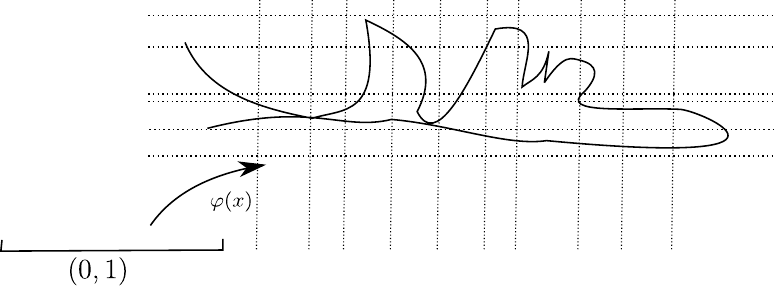}}
  \end{center}	
\caption{An illustration of the partitioning for a given functions $\varphi$. The partitioning is performed in the deformed configuration by means of a so-called \emph{good arrival grid}. }\label{grid}
}
\end{figure}

Second, we have a closer look at the intersection points of $\varphi((0,1))$ with the grid, we call one such point $z$. Because we chose the grid well, we know that the set $\varphi^{-1}(z)$ will have at most \emph{finitely} many points. Our first goal will be to tear these possible non-injectivities apart on a sufficiently small neighborhood of $z $; indeed, due to the choice of the grid, we know that $\varphi$ cannot oscillate too much in the vicinity of $z$, so that in particular, there is some distance to another cross-point or vertex of the grid in order to "fix" the non-injectivities. To do that, we use the approximating sequence of homeomorphisms. In fact, we have to shift the points on the grid everywhere, not just when non-injectivities happen, to obtain an injective function later. But we omit this in the illustration in Figure \ref{grid2}. The construction is detailed in \emph{Step 1 of the Proof of Theorem \ref{NotParticullarlyDeep}.}
\begin{figure}[h t p]
{\begin{center}
     {\includegraphics[width=0.9\textwidth]{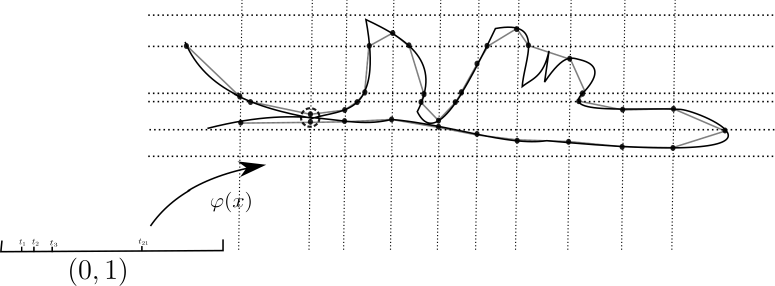}}
  \end{center}	
\caption{An illustration of injectifying on the grid in one point and piece-wise affine approximation which is illustrated in grey. }\label{grid2}
}
\end{figure}

Once torn apart, we just connect the individual intersecting points by a piecewise affine function. Here we need to be careful that this does not interfere with injectivity. But  it will come in handy that the injectivity discussion can actually be separated into the individual rectangles where we then can use topological arguments. This is detailed in \emph{Step 2 of the Proof of \ref{NotParticullarlyDeep}.}

\rone
Let us stress that it will turn out to vital at several points in the proof that the target of the mapping $\varphi$ is $\mathbb{R}^2$ and not $\mathbb{R}^3$. To mention one, we use topological argument restricted to the plane which roughly can be described as follows: If we are given a continuous injective function and whose graph passes though a rectangle (perhaps several times) and if we connect the two points where it enters and leaves the rectangle by a segment, the the segments obtained in this way will not cross.
\EEE

Let us now formalize the arguments. As outlined, we start the discussion with some preliminary results and the definition of a good arrival grid. First, we recall that in dimension one the weak convergence in $ W^{1,1}((0,1);\er^2)$ translates to uniform convergence. This is well-known (see e.g. the results in \cite{Le}) but we re-prove it here for the sake of completeness. 

\begin{lemma}\label{uniform}
Let $\phi,\phi_k\in W^{1,1}((0,1);\er^2)$ be such that $\phi_k \deb \phi$ in $W^{1,1}((0,1);\er^2)$. Then $\phi_k\sto \phi$ on $(0,1)$, where $\sto$ denotes the uniform convergence. 
\end{lemma}
\begin{proof}
The proof relies on the representation  
\begin{equation}
\label{1DREp}
\phi(x_1) = \phi(x_2) + \int_{x_1}^{x_2} \phi'(x) \mathrm{d} x    
\end{equation}
which holds for functions in $W^{1,1}((0,1);\er^2)$  for a.a. $x_1, x_2 \in (0,1)$. Now, as $\phi_k\rightharpoonup\phi$, we have in particular that the sequence is uniformly bounded in $W^{1,1}((0,1);\er^2)$. Further, we may assume (passing to a subsequence if necessary) that $\phi_k\to\phi$ strongly in $L^1((0,1);\er^2)$ and pointwise {\color{black}a.e.;} so we can pick $x_0\in (0,1)$ with $|\phi_k(x_0) - \phi(x_0)|\xrightarrow{k\to\infty}0$ and $\phi(x_0) \leq R$ for some $R > 0$ fixed. 

Setting $x_2=x_0$ in \eqref{1DREp} shows that the sequence $\{\phi_k\}_{k \in \mathbb{N}}$ is uniformly bounded. 
Further, as the sequence $\{\phi_k'\}_{k \in \mathbb{N}}$ is weakly convergent in $L^1((0,1);\er^2)$ it is also equiintegrable by the Dunford-Pettis theorem (see e.g. \cite[Theorem A.12]{Ri}) which gives the uniform integrability of these functions, i.e. for every $\epsilon>0$ there is a $\delta>0$ such that 
$$
\text{ for every }A\text{ with }\mathcal{L}^1(A)< \delta\text{ we have }\sup_{k \in \mathbb{N}} \int_{A}|\phi'_k(x)| \mathrm{d} x <\epsilon .
$$ 
Especially this implies the equicontinuity of $\{\phi_k\}_{k \in \mathbb{N}}$. Thus, the convergence is uniform due the Arzela-Ascoli theorem.

\end{proof}

 We now introduce the concept of a  good arrival grid that was actually introduced in \cite{DPP}. The following is an adaption to our context:

\begin{definition}[Good arrival grid]\label{good arrival grid}
	Let $\phi\in W^{1,1}((0,1);\er^2)$  be non-constant and let ${\color{black}L}>0$ be such that $\phi((0,1))$ is compactly contained in $Q(0,{\color{black}L})$ (the cube of side-length $2{\color{black}L}$ centered in $0$). Let $\delta>0$ be arbitrary and let the numbers 
	$$
	-{\color{black}L}=w_0<w_1< \cdots <w_N< w_{N+1}={\color{black}L}\text{ satisfy }w_{n+1}-w_n<\delta
	$$ 
	for every $0\leq n\leq N$. We say that
	\begin{equation}\label{defgag}
		\G = \bigcup_{n=0}^{N+1} \{w_n\}\times [-{\color{black}L},{\color{black}L}]\ \cup\ \bigcup_{n=0}^{N+1} [-{\color{black}L},{\color{black}L}] \times \{w_n\}\subset \overline{Q(0,{\color{black}L})}
	\end{equation}
	is a \emph{$\delta$-fine good arrival grid for $\phi$} if the set $F:=\{x\in(0,1):\phi(x)\in\G\}$ is finite and for every $p\in F$ it holds that 
	\begin{itemize}
		\item[$a)$] $\phi(p) \in \G\setminus \{(w_n, w_m); 0\leq n,m\leq N+1\}$, i.e. $\phi(p)$ is not a cross point of the grid $\G$
		\item[$b)$] there exists $\phi'(p)$,
		\item[$c)$] $|\langle\phi'(p), \vec{n}\rangle|>0$ where $\vec{n}$ is a unit vector perpendicular to the side of $\G$ containing $\phi(p)$.
	\end{itemize}
\end{definition}

\begin{figure}[h t p]
\vskip 190pt
{\begin{picture}(0.0,0.0) 
     \put(-150.2,0.2){\includegraphics[width=1.00\textwidth]{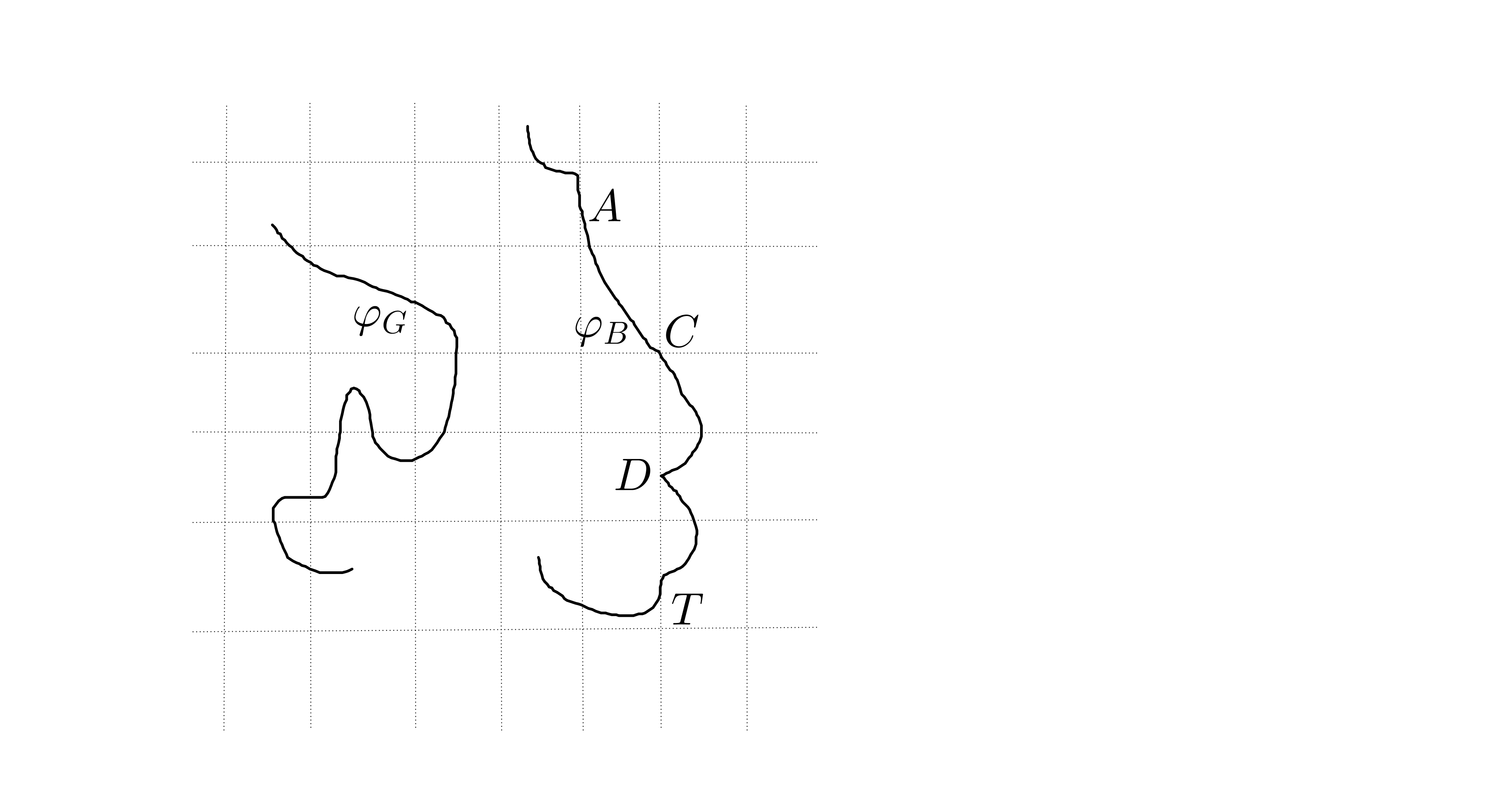}}
  \end{picture}
\vskip -20pt	
\caption{The grid above is a good arrival grid for $\varphi_G$. It is not a good grid for $\varphi_B$ for several reasons: around $A$ it intersects the grid in infinitely many points, it crosses a point $C$ which is the cross point of the grid, it does not have a derivative at $D$ and its derivative at $T$ is parallel to the grid. }\label{goodgrid}
}
\end{figure}

The existence of many such arrival grids was proven in~\cite[Lemma~3.6]{DPP} and the following is a variant of that claim adapted to our context.

\begin{lemma}\label{ArrivalGrid}
	Let $\phi\in W^{1,1}((0,1);\er^2)$  be non-constant  and let ${\color{black}L}>0$ be such that $\overline{\phi((0,1))}$ is compactly contained in $Q(0,{\color{black}L})$. For every $\delta >0$ there exists a $\delta$-fine good arrival grid $\G$ for $\phi$.
\end{lemma}

In the proof we will use the classical area formula which we recall here for convenience. Let $h:(0,1)\to\er$ be absolutely continuous and $A\subset (0,1)$ be a Borel set, then (see e.g. \cite[Theorem 3.65 for $\psi=\chi_A$]{Le})
\eqn{area}
$$
\int_A|h'(t)| \mathrm{d} t = \int_{\er} \H^0(\{t\in A:\ h(t) = z\}) \, \mathrm{d} z. 
$$
\begin{proof}[Proof of Lemma \ref{ArrivalGrid}]
	Call $\pi_1(x,y)  := x, \pi_2(x,y)  := y$ the projections to the axes. Let us recall that each $\phi\in W^{1,1}((0,1);\er^2)$ is absolutely continuous and thus $\pi_i\circ \phi$ are absolutely continuous so that the area formula applies to them and their derivatives exist a.e. Denote
	$$
		N_i = \{t\in (0,1); \text{ derivative }\phi'(t)\text{ does not exist}\} \cup \{t\in (0,1); (\pi_i\circ\phi)'(t) = 0\}.
	$$
	We can clearly fix a Borel set $\tilde{N_i}\supset N_i$ of the same measure. For a.e. $t\in \tilde{N_i}$ we  have $(\pi_i\circ\phi)'(t) = 0$ and hence
 we can use the area formula \eqref{area} to obtain 
	$$
	\mathcal{L}^1(\pi_i\circ\phi(N_i))\leq\mathcal{L}^1(\pi_i\circ\phi(\tilde{N}_i))\leq \int_{\er} \H^0(\{t\in \tilde{N}_i:\ \pi_i(\phi(t)) = z\}) \, \mathrm{d}z =\int_{\tilde{N}_i}(\pi_i\circ \phi)'(t)\; \mathrm{d} t=0.
	$$ 
	Further, since $\phi\in W^{1,1}((0,1);\er^2)$, the area formula applied to $\pi_i\circ\phi$ again gives
	$$
	\infty > \int_0^1|\phi'(s)|{\color{black} \, \mathrm{d} s} \geq \int_0^1|\pi_i\circ\phi'(s)|{\color{black} \, \mathrm{d} s} = \int_{-{\color{black}L}}^{\color{black}L} \H^0(\{t:\ \pi_i(\phi(t)) = z\})  {\color{black} \, \mathrm{d} z} 
	$$
	and so
	$$
		\H^0\Bigl(\phi((0,1))\cap \bigl(\{x\}\times \er\bigr)\Bigr) < \infty \ \text{and} \ \H^0\Bigl(\phi((0,1))\cap \bigl(\er\times\{y\}\bigr)\Bigr) < \infty
	$$
	for almost every $x,y\in \er$. We find an open set
	$$
		\begin{aligned}
			U \supset \, \big(\pi_1\circ\phi(N_1 ) \big) \, &\cup \, \big(\pi_2\circ\phi(N_2) \big) \, \cup \{x\in \mathbb{R}; \H^0(\phi((0,1))\cap (\{x\}\times \er)) = \infty\}\\
			& \cup \{y \in \er;   \H^0(\phi((0,1))\cap (\er\times\{y\})) = \infty\}
		\end{aligned}
	$$
	such that $\mathcal{L}^1(U)< \delta / 2$. Then every interval of length $\delta/2$ contains a point of $\er\setminus U$. We separate $[-{\color{black}L},{\color{black}L}] $ into $N$ intervals $I_n$ of length $\delta/2$ (there is no barrier to assume that ${\color{black}L}\delta^{-1}\in \en$). In each interval $I_n$ we choose a point $w_n$ from $I_n\setminus U$ arbitrarily. The distance of $w_n$ from his neighbors is at most $\delta$. By the choice of $U$ we satisfy the definition of a good arrival grid for $\phi$.
\end{proof}

We define the concept of the generalized segment, already introduced in \cite{DPP}, which will be useful throughout the proof of Theorem \ref{NotParticullarlyDeep}.

\begin{definition}[generalized segments]\label{generalized segments}
	Let $\mathcal{G} \subset Q(0,1)\subset \er^2$ be a grid (the finite union of horizontal and vertical lines). Let $K$ be a rectangle of the grid $\G$ (the closure of a component of $Q(0,1) \setminus \G$). Let $X\neq Y$ and $X,Y\in\partial K \subset \G$. 
 
 Given $\xi >0$ a small parameter, the \emph{generalized segment $[XY]$ with parameter $\xi$}  between $X$ and $Y$ in $K$ is defined as the standard segment $XY$ if the two points are not on the same side of $\partial K$; otherwise, $[XY]$ is the union of two segments of the form $XM$ and $MY$ where $M$ is the point inside $K$ whose distance from the side containing $X$ and $Y$ is $\xi|X-Y|/2$ and the projection of $M$ on the segment $XY$ is the mid-point of $XY$.
\end{definition}

	The following claim from \cite{CKR} is easy.

\begin{prop}\label{EverythingIDo} 
	Let $K\subset \er^2$ be a rectangle and let $a,b \in \partial K$. Let $S$ be a generalized segment from $a$ to $b$ in $K$ with parameter $\xi >0$ and let $\tilde{S} \subset S$ be a closed and connected subset of $S$. Then
	$$
	\H^1(\tilde S) \leq (1+\xi)\diam(\tilde{S}),
	$$
 {\color{black} where $\diam(\tilde{S}) = \sup\{|x-y|; x,y\in \tilde{S}\}$ is taken with respect to the Euclidean norm in $\mathbb{R}^2$}.
\end{prop}

We are now at the position to prove Theorem \ref{NotParticullarlyDeep}:

\begin{proof}[Proof of Theorem \ref{NotParticullarlyDeep}]
Condition $(1)$ implies condition $(2)$ immediately. By Lemma~\ref{uniform}, condition $(2)$ implies condition $(3)$. Therefore we only need to prove that $(3)$ implies $(1)$. We proceed in several steps and prove at the same time that the approximating sequence can be chosen as piece-wise affine.

{\underline{Step 1 - Construction of an injective approximation $\hat{\varphi}_{\delta}$ using the grid $\G_{\delta}$:}} \\
Let $\phi \in W^{1,p}((0,1), \er^2)$ for some $p\in[1, \infty)$. By condition $(3)$ we can find equicontinuous $\tilde{\phi}_{ \delta }$ so that $\tilde{\phi}_{ \delta }\sto \phi$ as $ \delta \to 0 $. We first construct some preliminary  injective continuous maps $\hat{\phi}_{\delta}$  using  $\tilde{\phi}_{\delta}$ to essentially create an injective approximation in the grid $\G_{\delta}$. Later we will use these functions to define  $\phi_{\delta}$ for which we will have $\phi_{\delta} \to \phi$ in $W^{1,p}((0,1), \er^2)$ as $\delta\to 0^+$.

Let $0<\delta<\frac{1}{2}$ be fixed. Let $\G_{\delta}$ be a $\delta$-fine arrival grid for $\phi$ and let the numbers $w_n$ be defined as in \eqref{defgag}. By {\color{black} the }definition of {\color{black} a } good arrival grid, we know that the set 
$$
F_{\delta}:=\{t\in(0,1): \phi(t)\in \G_{\delta}\}
$$ 
is finite. 
 We now start to show that {\color{black} we can find sufficiently small neighborhoods} {\color{black}$(t-\eta, t+\eta)$} and $B(\varphi(t), \varepsilon)$ for $t \in F_\delta$ so that the intervals {\color{black}$(t-\eta, t+\eta)$} are pair-wise disjoint and the balls $ B(\varphi(t), \varepsilon)$ are pair-wise disjoint and do not intersect the corners $(w_n,w_m)$. Moreover, we want to take $\varepsilon$ so small that, roughly speaking, an affine interpolation connecting any point in $B(t_1,\eta)$ and $B(t_2,\eta)$ will be a good approximation (in the  $W^{1,p}((0,1), \er^2)$-norm) for the original functions $\varphi$. Of course, having then the freedom to select the points to construct the approximation, we will do that carefully to respect injectivity.  

At each point $t\in F_{\delta}$ it holds that $|\phi'(t)|>0$ and $\phi'(t)$ is not parallel to the side of $\G_{\delta}$ containing $\phi(t)$. Therefore we can find a $v>0$ {\color{black} and} $\eta  =\eta(\delta) >0$ and for every $t\in F_{\delta}$ a unit vector $\vec{n}_t$ perpendicular to $\G_{\delta}$ at $\phi(t)$ such that
	\begin{equation}\label{cross}
		\langle \phi(t+h)-\phi(t), \vec{n}_t\rangle \geq hv
	\end{equation}
for all $h\in [-\eta, \eta]$. 
By the differentiability of $\phi$ at $t$ we can also assume that 
$$
\bigl|\phi(t+h)-\phi(t)-\phi'(t)h\bigr|\leq |h| v
$$
for every $h\in [-\eta, \eta]$. 
Moreover,  by choosing $\eta$ smaller if necessary,  we may assume that the intervals $[t-\eta, t+\eta]$ are pairwise disjoint for $t\in F_{\delta}$. 
Call
$$
	\begin{aligned}
		\epsilon_1 & = \epsilon_1(\delta) := \min\{\dist((w_n, w_m), \phi(t));\,  t\in F_{\delta},  1\leq n,m\leq N\},\\
		\epsilon_2 & = \epsilon_2(\delta) := \min\big\{\varepsilon_1, \min\big\{|\phi(t)- \phi(s)|;\, s,t \in F_{\delta}, \phi(t)\neq \phi(s) \big\} \big\},\\
	\end{aligned}
$$
 where, by definition, $\varepsilon_2 \leq \varepsilon_1$.
Then the balls $B(z,\epsilon_2)$, {\color{black}$z\in \phi(F_{\delta})$,} are pairwise disjoint and do not intersect corners of the grid.  By further shrinking the balls at the intersection of $\phi((0,1))$ with $\G_{\delta}$, if necessary, we can take 
$$
\epsilon_3 <\min \{\varepsilon_2, \eta v\}, 
$$
 so it holds that
$$
	\phi(s) \in B(z, \epsilon_3) \text{ for some } z\in \phi(F_{\delta})  \Leftrightarrow s\in [t-\eta, t+\eta] \text{ for some } t\in F_{\delta} \text{ and } \phi(t) = z.
$$
 This means that inside the balls $B(z, \epsilon_3)$, we can only see the images of the intervals $[t-\eta, t+\eta]$, with $t \in F_\delta$.

At last, we may still shrink the neighborhoods $B(z, \varepsilon_3)$ so that we obtain a good approximation also on the discrete derivatives:  Thanks to the fact that  $F_{\delta}$ is finite, we can fix a number  $\nu>0$ such that for every $t_1,t_2\in F_{\delta}$ with $\phi(t_1)\neq\phi(t_2)$ we have
\eqn{defnu}
$$
|\phi(t_1)-\phi(t_2)|+4\nu<(1+\delta)|\phi(t_1)-\phi(t_2)|. 
$$
We can also assume that our $\nu$ is so small that for every distinct $t_1,t_2\in F_{\delta}$ we have 
\eqn{defnu2}
$$
(t_2-t_1)^{1-p}\nu^p\# F_{\delta}\leq \delta, 
$$
where $\# F_{\delta}$ is the number of elements in $F_{\delta}$.  Choosing now $\varepsilon < \min\{\varepsilon_3, \nu, \frac{\delta}{\# F_{\delta}}\}$ (here the last term will be needed in the $W^{1,1}$-case) yields  balls sufficiently small for our needs. 

 Now, we want to find a {\color{black}suitable continuous injective approximant $\tilde{\varphi}_{\delta}$,} which we will later use to construct the strong injective approximation. What we will need is to take this first approximative function near enough to $\varphi$ so that we can assure that it crosses the grid \emph{only} inside the constructed balls $B(z,\varepsilon)$,{\color{black} $z\in \phi(F_{\delta})$}. To ensure that, we proceed as follows:  
For each $t\in F_{\delta}$ we define $t^-, t^+ \in [t-\eta, t+\eta]$ as the smallest and largest argument respectively such that $\phi(t^{\pm})\in \partial B(\phi(t), \epsilon)$ (see Figure \ref{defphi}). 
Notice by \eqref{cross} that $\phi(t^-)$ and $\phi(t^+)$ lie on different components of $\partial B(\phi(t), \epsilon)\setminus \G_{\delta}$.

\begin{figure}[h t p]
\vskip 200pt
{\begin{picture}(0.0,0.0) 
     \put(-150.2,0.2){\includegraphics[width=0.70\textwidth]{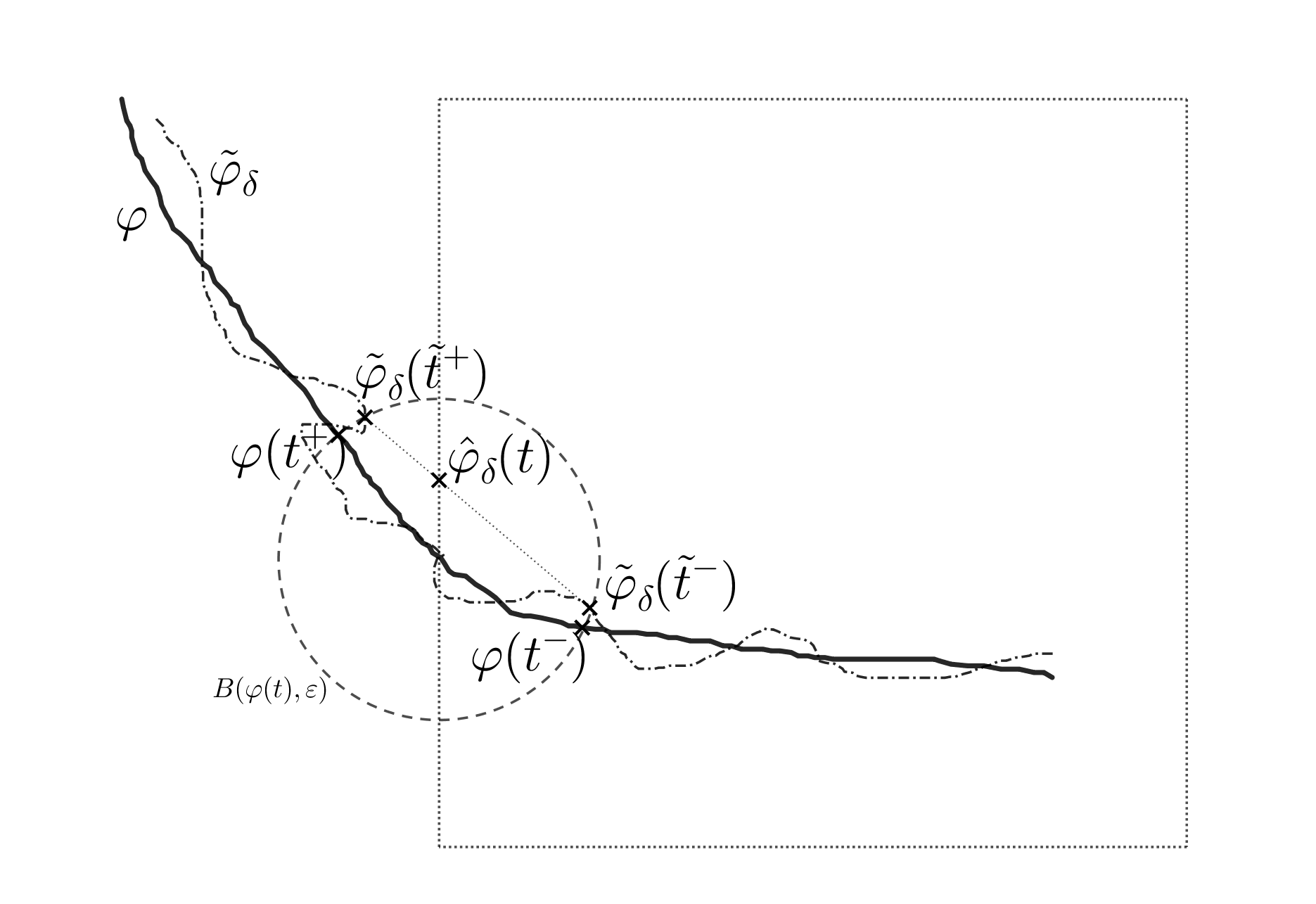}}
  \end{picture}
\vskip -10pt	
\caption{An illustration of the definition of the points $\tilde{\phi}_{\delta}(\tilde{t}^{\pm})$ and $\hat{\phi}(t)$ in $B(\varphi(t),\epsilon)$. The curve $\phi$ is drawn by a full line while the curve and $\tilde{\varphi}_{\delta}$ are dash-dotted. The segment $\tilde{\phi}_{\delta}(t^-)\tilde{\phi}_{\delta}(t^+)$ is also indicated in the figure.}\label{defphi}
}
\end{figure}

Now we choose 
$$
\tepsilon:=\frac{1}{2}\dist\Bigl(\phi\bigl((0,1)\setminus\bigcup_{t\in F_{\delta}}(t^-,t^+)\bigr),\G_{\delta}\Bigr).
$$
By condition $(3)$, we may find a map   $\tilde{\phi}_{\delta}:(0,1)\to \er^2$ that is injective, continuous and $\| \tilde{\phi}_{\delta}\EEE - \phi\|_{ C^0((0,1); \mathbb{R}^2)}< \tepsilon$, i.e. it crosses the grid exactly in the balls $B(z,\varepsilon)$,{\color{black} $z\in \phi(F_{\delta})$}.

We now use that function $\tilde{\phi}_{\delta}$ to ``tear apart'' the non-injectivities on the grid: \EEE For each $t\in F_{\delta}$ we define $\tilde{t}^-, \tilde{t}^+ \in [t-\eta, t+\eta]$ as the smallest and largest argument respectively such that $\tilde{\phi}_{\delta}(\tilde{t}^{\pm})\in \partial B(\phi(t), \epsilon)$ (see Figure \ref{defphi}). By 
$$
\| \tilde{\phi}_{\delta} - \phi\|_{ C^0((0,1); \mathbb{R}^2)}< \tepsilon\text{ we obtain }| \tilde{\phi}_{\delta}(\tilde{t}^{\pm})-\phi(t^{\pm})|<\tepsilon.
$$ 
We know that $\phi(t^-)$ and $\phi(t^+)$ lie on different components of $\partial B(\phi(t), \epsilon)\setminus \G_{\delta}$ and hence by the definition of $\tepsilon$ we obtain 
that also $ \tilde{\phi}_{\delta}(\tilde{t}^-)$ and $ \tilde{\phi}_{\delta}\EEE(\tilde{t}^+)$ lie on different components of $\partial B(\phi(t), \epsilon)\setminus \G_{\delta}$. 
Therefore the segment $ \tilde{\phi}_{\delta}(t^-) \tilde{\phi}_{\delta} (t^+)$ {\color{black}starts on one side of $\G_{\delta}\cap B(\phi(t), \epsilon)$ and ends on the other side of $\G_{\delta}\cap B(\phi(t), \epsilon)$ inside $B(\phi(t), \epsilon)$. Thus $ \tilde{\phi}_{\delta}(t^-) \tilde{\phi}_{\delta} (t^+)$} intersects $\G_{\delta}$ exactly once, call this point $\hat{\phi}_{\delta}(t)$. 

Note that we obviously have
\eqn{phiclose}
$$
|\phi(t)-\hat{\phi}_{\delta}(t)|<2\epsilon<2\nu. 
$$

Now we use $ \tilde{\phi}_{\delta}$ to define the intermediary injective continuous $\hat{\phi}_{\delta}:(0,1)\to \er^2$. 
On $[\tilde{t}^-, \tilde{t}^+]$ we define $\hat{\phi}_{\delta}$ {\color{black} by gluing the linear map that sends $[\tilde{t}^-,t]$ onto the segment $ \tilde{\phi}_{\delta}(\tilde{t}^-)\hat{\phi}_{\delta}(t)$ to the map that sends $[t,\tilde{t}^+]$ onto the segment $\hat{\phi}_{\delta}(t)\tilde{\phi}_{\delta}(\tilde{t}^+)$}. We do this for all $t\in F_{\delta}$. On the set $(0,1)\setminus \bigcup_{t\in F_{\delta}}[\tilde{t}^-,\tilde{t}^+]$ we put $\hat{\phi}_{\delta} =  \tilde{\phi}_{\delta}$. 

\begin{figure}[h t p]
\vskip 160pt
{\begin{picture}(0.0,0.0) 
     \put(-150.2,0.2){\includegraphics[width=0.80\textwidth]{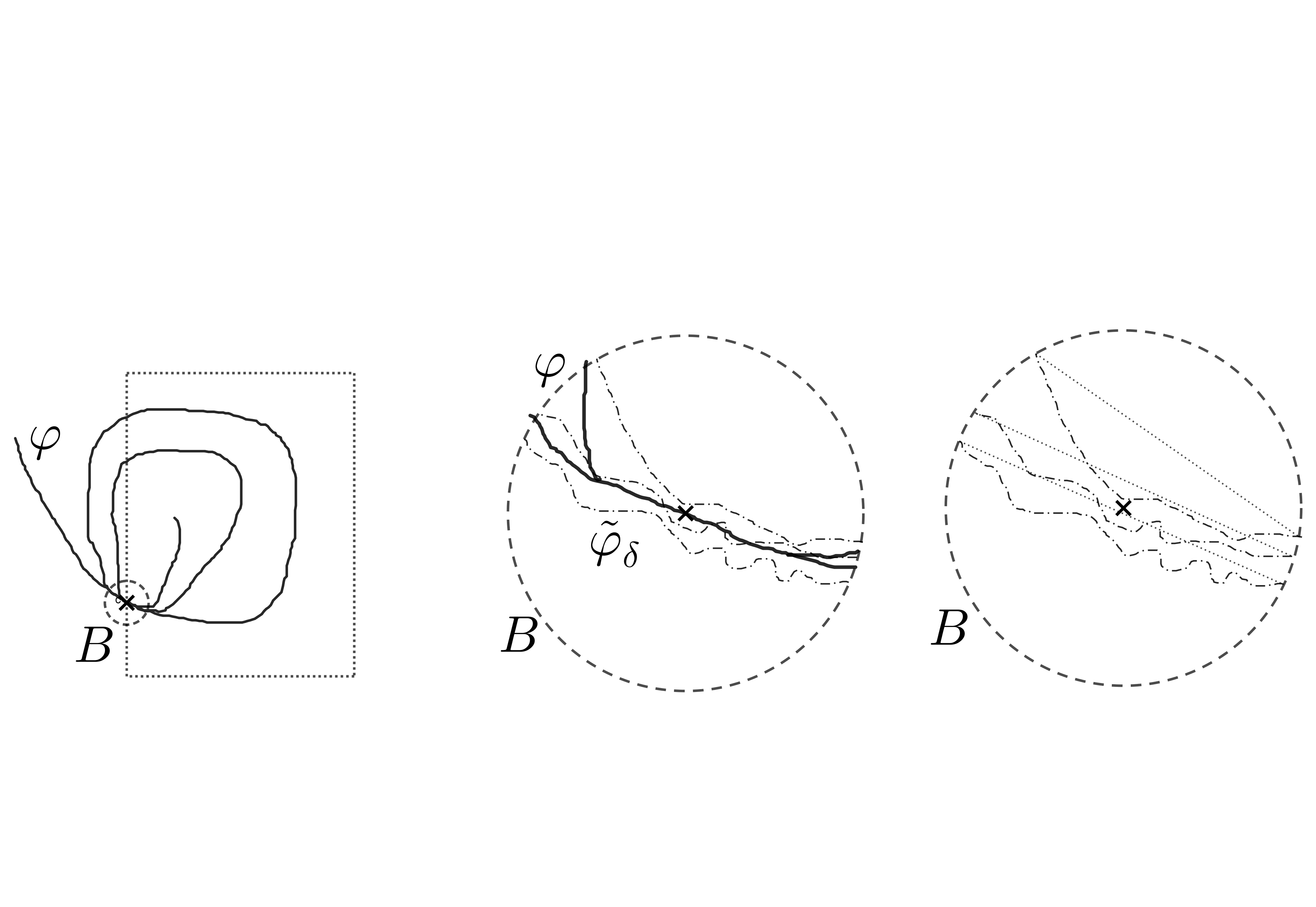}}
  \end{picture}
\vskip -40pt	
\caption{Definition of $\hat{\phi}_{\delta}$ inside of ball $B:=B(\varphi(t),\epsilon)$ such that $\varphi$ passes through $\varphi(t)$ three times. 
the curve $\phi$ is denoted by a full line while  $ \tilde{\phi}_{\delta}$ is dash-dotted. On the left, we can see how $\varphi$ passes through $\varphi(t)$ three times, and in the middle picture, we can see in detail how three injective pieces of $ \tilde{\phi}_{\delta}$ approximate $\phi$ inside $B$. On the right picture, we can see that three pieces of $ \tilde{\phi}_{\delta}$ are injective inside $B$ and hence their replacement by $\hat{\varphi}_{\delta}$ (see the three-dotted segments there) are also injective there and do not intersect.}\label{defhat}
}
\end{figure}

Clearly $\hat{\phi}_{\delta}$ is continuous. Thanks to the injectivity of $\ \tilde{\phi}_{\delta}$ it is not hard to observe that $\hat{\phi}_{\delta}$ is injective. 
Indeed, we know that $B(z,\epsilon)$, $z\in F_{\delta}$, are pairwise disjoint and outside of these balls our $\hat{\phi}_{\delta}$ is injective since it agrees with $ \tilde{\phi}_{\delta}$.  Moreover, by the choice of the numbers $\tilde{t}_i^-, \tilde{t}_i^+$, we know that $\tilde{\varphi}_\delta((0,1)\setminus \bigcup_i (\tilde{t}_i^-, \tilde{t}_i^+))$ is disjoint with all the balls $B(z,\varepsilon)$. Thus, it is sufficient to verify the injectivity on one fixed ball $B(z,\epsilon)$, $z\in F_{\delta}$. There,  we have finitely many $t_1,\hdots,t_k\in F_{\delta}$ such that $\phi(t_i)=z$ and each such 
$t_i$ gives us one {\color{black} affine} segment $ \tilde{\phi}_{\delta}(\tilde{t}_i^-) \tilde{\phi}_{\delta}(\tilde{t}_i^+)$ inside $B(z,\epsilon)$ (see Figure \ref{defhat}). Since the curve $ \tilde{\phi}_{\delta}$ was injective, it follows that parts of $\tilde{\phi}_{\delta}$ between points {\color{black}$ \tilde{\phi}_{\delta}(\tilde{t}_i^-)$ and $\tilde{\phi}_{\delta}\EEE(\tilde{t}_i^+)$} do not intersect and hence a simple topological argument valid in dimension two gives that segments $\tilde{\phi}_{\delta}\EEE(\tilde{t}_i^-)\tilde{\phi}_{\delta}(\tilde{t}_i^+)$ also do not intersect (see Figure \ref{defhat}). {\color{black} To wit, each component of $\tilde{\phi}_{\delta}(0,1) \cap B(z,\epsilon)$ seperates the disk $B(z,\epsilon)$ into two domains, which we will call `north' if the closure of that component contains a point $N\in \partial B(z,\epsilon) \setminus \tilde{\phi}_{\delta}(0,1)$ and a `south' otherwise. One of the components of $\tilde{\phi}_{\delta}(0,1) \cap B(z,\epsilon)$ is north of another exactly when its endpoints are north of the endpoints of the other and this is exactly when the corresponding segments are aligned in the same way.} 
Thus,  injectivity of $\hat{\phi}_{\delta}$ follows. Note that $\hat{\phi}_{\delta}^{-1} (\G_{\delta})= F_{\delta}$ so the set $F_{\delta}$ is the same for $\hat{\phi}_{\delta}$ as it was for the original $\phi$. 

{\underline{Step 2 - Construction of injective $ {\varphi}_{\delta} $ and approximation in $W^{1,p}((0,1); \mathbb{R}^2)$ for $p>1$:}} 

Now we define $\phi_{\delta}$  basically as an affine interpolation between the points $\hat{\phi}_{\delta}(t_i), t_i \in F_\delta$, when care is needed if two  $\hat{\phi}_{\delta}(t)$'s lie on the same side. In more detail,  for each $t\in F_{\delta}$ we put $\phi_{\delta}(t) = \hat{\phi}_{\delta}(t)$. For all neighboring pairs $t_1,t_2 \in F_{\delta}$ such that $\phi_{\delta}(t_1)=\hat{\phi}_{\delta}(t_1)$ and $\phi_{\delta}(t_2)=\hat{\phi}_{\delta}(t_2)$ do not lie on the same side of $\G_{\delta}$ we define $\phi_{\delta}$ on $[t_1,t_2]$ as the  affine function which equals $\hat{\phi}_{\delta}$ at $t_1$ and $t_2$. It is not difficult to see that if we have two segment $\phi_{\delta}(t_1)\phi_{\delta}(t_2)$ and 
$\phi_{\delta}(\tilde{t}_1)\phi_{\delta}(\tilde{t}_2)$ of this type that they do not intersect as the original curves $\hat{\phi}_\delta([t_1,t_2])$ and $\hat{\phi}_\delta([\tilde{t}_1,\tilde{t}_2])$ do not intersect as $\hat{\phi}$ is injective. Indeed, either the two segments lie in the same rectangle of the grid, in which case we may rely on the same topological argument as above; or they lie in two different rectangles and thus are disjoint. 
In case our neighboring pair $t_3,t_4 \in F_{\delta}$ lies on the same segment 
in the grid $\G_{\delta}$ we connect ${\phi}_{\delta}(t_3)$ and ${\phi}_{\delta}(t_4)$ by a generalized segment $[{\phi}_{\delta}(t_3){\phi}_{\delta}(t_4)]$ with parameter 
$0<\xi<\frac{\delta}{\# F_{\delta}}$  in such a way that the generalized segment lies in the same rectangle of the grid as the original {\color{black} curve $\hat{\varphi}([t_3,t_4])$.} 

In this way, the verification of whether the introduction of generalized segments preserves injectivity reduces to one rectangle of the grid.
As we have only finitely many points in $F_{\delta}$, this can be assured by choosing  {\color{black}$\xi>0$} sufficiently small so that the newly introduced generalized segments do not intersect each other or the previous segments. 


Having chosen $\xi$ we define $\phi_{\delta}$ on the generalized segments via the constant speed parametrization.  It remains to extend $\phi_{\delta}$ onto the intervals $(0,t_{\text{min}}]$ and $[t_{\text{max}}, 1)$, where $t_{\text{min}}, t_{\text{max}}$ are the largest and smallest $t\in F_{\delta}$ (so the fist and last crossing point with the grid). We can do that by, for example, connecting $\varphi(0)$ and  $\phi_{\delta}(t_{\text{min}})$ as well as $\phi_{\delta}(t_{\text{max}})$ and $\varphi(1)$ by a segment. 
		
 Let us notice that by construction, we have that for any $K = K_{n,m} = [w_n,w_{n+1}]\times[w_m, w_{m+1}]$ (a closed rectangle formed by the grid $\G_{\delta}$) we have that $\phi_{\delta}(s)\in K$ exactly if $\phi(s)\in K$. 

Therefore 
$$
\|\phi_{\delta} - \phi\|_{L^{\infty}((0,1); \mathbb{R}^2)}\leq \diam K < 4\delta 
\text{ and thus also }
\|\phi_{\delta} - \phi\|_{L^{p}((0,1); \mathbb{R}^2)}<4\delta.
$$
Note that this uniform convergence works also for $p=1$. 

Now we prove that $\phi_{\delta}' \to \phi'$ in $L^{p}((0,1); \mathbb{R}^2)$. Let us have a pair of neighbors $t_1,t_2\in F_{\delta}$. 
We know that for $t\in F_{\delta}$ we have $\phi_{\delta}(t)=\hat{\phi}_{\delta}(t)$ and hence \eqref{phiclose} and \eqref{defnu} imply that 
$$
|\phi_{\delta}(t_2)-\phi_{\delta}(t_1)|\leq |\phi(t_2)-\phi(t_1)|+4\nu<(1+\delta) |\phi(t_2)-\phi(t_1)|. 
$$
{\color{black} In the following first equality we use the fact that the parametrization is  of constant speed. Then, using Proposition~\ref{EverythingIDo} we have, for each pair of neighbors $t_1,t_2\in F_{\delta}$ with $\phi(t_1)\neq \phi(t_2)$, that
\eqn{longer}
$$
\begin{aligned}
	\int_{t_1}^{t_2}|\phi_{\delta}'|^p &= \int_{t_1}^{t_2}\Big(\oint_{t_1}^{t_2}|\phi_{\delta}'(s)|ds\Big)^p \mathrm{d}w \\
	&\leq \int_{t_1}^{t_2}\Big(\frac{1}{t_2-t_1}(1+\delta)\big|\phi_{\delta}(t_2)-\phi_{\delta}(t_1)\big|\Big)^p \mathrm{d}w\\
	&\leq \int_{t_1}^{t_2}\Big(\frac{1}{t_2-t_1}(1+\delta)^2\big|\phi(t_2)-\phi(t_1)\big|\Big)^p \mathrm{d}w\\
	&\leq (1+\delta)^{2p}  (t_2-t_1)\Bigl(\oint_{t_1}^{t_2}|\phi'(s)|\mathrm{d}s\Bigr)^p\\
	&\leq (1+\delta)^{2p} \int_{t_1}^{t_2}|\phi'(s)|^p \mathrm{d}s.
\end{aligned}
$$}
In case $\phi(t_1)=\phi(t_2)$ we can estimate similarly and use \eqref{phiclose} to obtain 
\eqn{longer2}
$$
\begin{aligned}
	\int_{t_1}^{t_2}|\phi_{\delta}'|^p 
	&\leq \int_{t_1}^{t_2}\Big(\frac{1}{t_2-t_1}(1+\delta)|\phi_{\delta}(t_2)-\phi_{\delta}(t_1)|\Big)^p \mathrm{d}w\\
	&\leq \int_{t_1}^{t_2}\Big(\frac{1}{t_2-t_1}(1+\delta)4\eta\Big)^p \mathrm{d}w\\
	&\leq C(p)(t_2-t_1)^{1-p} \eta^p.\\
\end{aligned}
$$

Summing over neighbors in $F_{\delta}$ we have using \eqref{defnu2} that
	$$
	\begin{aligned}
		\int_{0}^{1}|\phi_{\delta}'(s)|^p \mathrm{d}s
		\leq & (1+\delta)^{2p} \int_{0}^{1}|\phi'(s)|^p \mathrm{d}s +C\delta.
	\end{aligned}
	$$ 

	For $p\in (1,\infty)$ the fact that $\|\phi_{\delta}'\|_p$ is bounded 
 means that for any infinitesimal sequence $\delta_k \to 0$ we can find a subsequence  such that $\phi_{\delta_{k_j}} \deb \phi$ in $W^{1,p}((0,1); \mathbb{R}^2)$. The inequality above implies that $\|\phi'_{\delta_{k_j}}\|_{L^p((0,1); \mathbb{R}^2)}$ is converging to $\|\phi'\|_{L^p((0,1); \mathbb{R}^2)}$ and hence strict convexity of $L^p$, $1<p<\infty$, implies using Clarkson's inequalities in $L^p$ that we have also strong convergence in $W^{1,p}$. 

{\underline{Step 3 - Approximation in $W^{1,1}((0,1); \mathbb{R}^2)$:}} 

Now we prove that $\phi_{\delta}' \to \phi'$ in $L^{1}((0,1); \mathbb{R}^2)$. We will need to do this differently to above as we cannot use uniform convexity of the space. Choose $\alpha \in (0,\tfrac{1}{4})$ and for the choice of $\alpha$ find $\beta>0$ such that 
$$
\text{ for every } E\subset(0,1)\text{ with }\mathcal{L}^1(E)< \beta\text{ we have }\int_E|\phi'|< \alpha.
$$
Because, for almost every $t\in (0,1)$ it holds that
$$
\lim_{r\to0^+} \oint_{t-r}^{t+r} |\phi'(s)-\phi'(t)|  \mathrm{d}s = 0
$$
we can find small parameters $\lambda, r\in (0,\alpha)$ such that $\mathcal{L}^1((0,1)\setminus H_{\lambda, r})<\beta$ where
$$
	H_{\lambda, r} := \Big\{ t\in(0,1); \ \exists \phi'(t),  \ |\phi'(t)|\in \{0\}\cup(\lambda,\infty), \sup_{\rho\in(0,r)}  \oint_{t-\rho}^{t+\rho} |\phi'(s)-\phi'(t)| \mathrm{d}s \leq \tfrac{1}{8}\lambda \Big\}.
$$
We also denote $H_{\lambda, r}^0 =H_{\lambda, r}\cap\{t;\phi'(t) = 0\}$ and $H_{\lambda, r}^+ = H_{\lambda, r}\setminus H_{\lambda, r}^0$. We show that for any 
\eqn{defdelta}
$$
0<\delta  < \frac{1}{8}\lambda r
$$
we have a good estimate of $\int |\phi'-\phi'_{\delta}|$.

We sum \eqref{longer} for $p=1$ over all intervals $(t_1,t_2)\subset (0,1)\setminus H_{\lambda,r}^+$ and we obtain using $\phi'=0$ on 
$H_{\lambda, r}\setminus H_{\lambda,r}^+$ and $\mathcal{L}^1((0,1)\setminus H_{\lambda, r})<\beta$ that
\eqn{Love}
$$
\begin{aligned}
		\sum_{(t_1,t_2)\subset (0,1)\setminus H_{\lambda,r}^+}&\int_{t_1}^{t_2}|\phi'(s)-\phi_{\delta}'(s)| \mathrm{d}s\leq \sum_{(t_1,t_2)\subset (0,1)\setminus H_{\lambda,r}^+}\int_{t_1}^{t_2}\bigl(|\phi'(s)|+|\phi_{\delta}'(s)|\bigr) \mathrm{d}s\\
		&\leq (1+(1+\delta)^2)\sum_{(t_1,t_2)\subset (0,1)\setminus H_{\lambda,r}^+}\int_{t_1}^{t_2}|\phi'(s)| \mathrm{d}s\\
		&\leq (1+(1+\delta)^2)\sum_{(t_1,t_2)\subset (0,1)\setminus H_{\lambda,r}}\int_{t_1}^{t_2}|\phi'(s)| \mathrm{d}s\leq C\alpha.\\
		\end{aligned}
$$

It remains to consider intervals $[t_1,t_2]$, $t_1,t_2\in F_{\delta}$, such that there is a $t\in(t_1,t_2)\cap H_{\lambda, r}^+$. The first step is to prove using \eqref{defdelta} that  $[t_1,t_2]$ is entirely contained in $[t-r, t+r]$. Using $t\in H_{\lambda, r}^+$ and definition of $H_{\lambda, r}^+$ we have that
$$
	|\phi(t+r) - \phi(t) - \phi'(t)r| = \Big|  \int_{t}^{t+r}\bigl(\phi'(s) - \phi'(t)\bigr)\mathrm{d}s  \Big| \leq \int_{t-r}^{t+r}|\phi'(s) - \phi'(t)|ds\leq \tfrac{1}{8}\lambda r 
$$
and since $|\phi'(t)r| \geq \lambda r$ we can use \eqref{defdelta} to obtain 
$$
	|\phi(t+r) - \phi(t) |  \geq \frac{7}{8}\lambda r \geq 7\delta > \diam K
$$
and so  $\phi(t+r)$ lies in a different rectangle $K$ of the grid $\G_{\delta}$ than does $\phi(t)$. The same holds for $\phi(t-r)$. In other words $[t_1,t_2]$ is entirely contained in $[t-r, t+r]$.

Call $\rho:=\max\{|t-t_1|,|t-t_2|\}< |t_1-t_2|$. Using $t\in H_{\lambda, r}^+$ and definition of $H_{\lambda, r}^+$ we estimate $\phi'$ as follows
$$
	\int_{t_1}^{t_2} |\phi'(s)-\phi'(t)|\; \mathrm{d}s \leq \int_{t-\rho}^{t+\rho} |\phi'(s)-\phi'(t)| \mathrm{d}s \leq 2\rho\frac{1}{8}\lambda=  2|t_2-t_1|\frac{1}{8}\lambda 
$$
and summing over these intervals $t_1,t_2\in F_{\delta}$ containing some $t\in H_{\lambda , r}^+$ we get
$$
\sum_{(t_1,t_2)\cap H_{\lambda , r}^+\neq \emptyset} \int_{t_1}^{t_2} |\phi'(s)-\phi'(t)|\; \mathrm{d}s \leq \frac{\lambda}{4} \leq \frac{\alpha}{4}. 
$$

Concerning $\phi_{\delta}'$ we have firstly that
\eqn{firstly}
$$
\Big| \frac{\phi(t_2) - \phi(t_1)}{t_2-t_1} - \phi'(t) \Big| \leq \oint_{t-\rho}^{t+\rho} |\phi'(s)-\phi'(t)| \mathrm{d}s \leq \frac{\lambda}{4} \leq \frac{\alpha}{4}.
$$
Our $\phi_{\delta}$ has constant speed parametrization and hence either for every $s\in(t_1,t_2)$ we have 
$$
\frac{\phi_{\delta}(t_2) - \phi_{\delta}(t_1)}{t_2-t_1} = \phi_{\delta}'(s)
$$
or we use generalized segment between $\phi_{\delta}(t_1)$ and $\phi_{\delta}(t_2)$ and then using Proposition \ref{EverythingIDo} we get 
$$
\Big| \frac{\phi_{\delta}(t_2) - \phi_{\delta}(t_1)}{t_2-t_1} - \phi_{\delta}'(s) \Big| \leq 
\xi \Big| \frac{\phi_{\delta}(t_2) - \phi_{\delta}(t_1)}{t_2-t_1}\Big|. 
$$
Using \eqref{phiclose} we thus obtain in both cases
\eqn{secondly}
$$
	\Big| \frac{\phi(t_2) - \phi(t_1)}{t_2-t_1} - \phi_{\delta}'(s) \Big| \leq \frac{8\epsilon}{t_2-t_1}+\xi \Big| \frac{\phi(t_2) - \phi(t_1)}{t_2-t_1}\Big| .
$$
Together with $t\in H_{\lambda, r}^+$, the definition of $H_{\lambda, r}^+$, \eqref{firstly}, \eqref{secondly}, $\epsilon<\frac{\delta}{\# F_{\delta}}$ and $\xi<\frac{\delta}{\# F_{\delta}}$ this gives
$$
\begin{aligned}
  \sum_{(t_1,t_2)\cap H_{\lambda , r}^+\neq \emptyset}\int_{t_1}^{t_2} |\phi'_{\delta}(s)-\phi'(s)|\; \mathrm{d}s
  \leq& \sum_{(t_1,t_2)\cap H_{\lambda , r}^+\neq \emptyset}\int_{t_1}^{t_2} |\phi'(t)-\phi'(s)|\; \mathrm{d}s+\\
  &+\sum_{(t_1,t_2)\cap H_{\lambda , r}^+\neq \emptyset}\int_{t_1}^{t_2} \Bigl|\phi'(t)-\frac{\phi(t_2) - \phi(t_1)}{t_2-t_1}\Bigr|\;\mathrm{d}s\\
  &+\sum_{(t_1,t_2)\cap H_{\lambda , r}^+\neq \emptyset}\int_{t_1}^{t_2} \Bigl|\phi_{\delta}'(s)-\frac{\phi(t_2) - \phi(t_1)}{t_2-t_1}\Bigr|\; \mathrm{d}s\\
  \leq &\sum_{(t_1,t_2)\cap H_{\lambda , r}^+\neq \emptyset}\frac{\lambda}{8}(t_2-t_1)+\sum_{(t_1,t_2)\cap H_{\lambda , r}^+\neq \emptyset}\frac{\alpha}{4}(t_2-t_1)\\
  &+\sum_{(t_1,t_2)\cap H_{\lambda , r}^+\neq \emptyset}\bigl(8\epsilon+\xi 2\|\phi\|_{L^\infty((0,1); \mathbb{R}^2)}\bigr)\\
	 \leq & \frac{\alpha}{8}+\frac{\alpha}{4}+8\delta+\delta 2\|\phi\|_{L^\infty((0,1); \mathbb{R}^2)}.\\
\end{aligned}	
$$
In combination with \eqref{Love} this implies that $\phi_{\delta} \to \phi$ in $W^{1,p}((0,1); \mathbb{R}^2)$ as  $\delta\to 0_+$. 
\end{proof}

{\bf Acknowledgment. }
    We would like to thank two anonymous reviewers for their careful reading of the manuscript and for many helpful remarks and suggestions.

\end{document}